\documentclass[11pt,a4paper,reqno,oneside]{amsart}

\usepackage[T1]{fontenc}
\usepackage[utf8]{inputenc}
\usepackage[dvipsnames]{xcolor}
\usepackage{lmodern}
\usepackage{microtype}
\usepackage{geometry}

\usepackage{amsthm}
\usepackage{amsmath}
\usepackage{amsfonts} 
\usepackage{amssymb}
\usepackage{mathtools}
\usepackage{esint}
\usepackage{stmaryrd}

\usepackage[final,unicode=true, pdflang={en-GB}, pdfusetitle,allbordercolors=RoyalBlue, pdfborder={0 0 .4}]{hyperref}

\let\C\undefined

\usepackage[inline]{enumitem}
\usepackage{moreenum}
\usepackage[abbrev]{amsrefs}
\usepackage{constants}
\usepackage{thmtools}
\usepackage[capitalize]{cleveref}
 
\numberwithin{equation}{section}

\newtheorem{proposition}{Proposition}[section]
\newtheorem{theorem}[proposition]{Theorem}
\newtheorem{lemma}[proposition]{Lemma}
\newtheorem{corollary}[proposition]{Corollary}
\theoremstyle{definition}
\newtheorem{definition}[proposition]{Definition}
\newtheorem{remark}[proposition]{Remark}

\DeclareMathOperator{\dist}{dist}
\DeclareMathOperator{\tr}{tr}

\newcommand{\restr}[1]{\vert_{#1}}
\newcommand{\Deriv}{\mathrm{D}}
\newcommand{\defeq}{\coloneqq}

\newcommand{\Nset}{\mathbb{N}}
\newcommand{\Zset}{\mathbb{Z}}
\newcommand{\Rset}{\mathbb{R}}
\newcommand{\Sset}{\mathbb{S}}
\newcommand{\RPset}{\mathbb{RP}}

\newcommand{\Bset}{\mathbb{B}}

\newcommand{\dif}{\,\mathrm{d}}

\newcommand{\compose}{\,\circ\,}
\newcommand{\manifold}[1]{\mathcal{#1}}
\newcommand{\lifting}[1]{\smash{\widetilde{#1}}}

\DeclarePairedDelimiter{\brk}{(}{)}
\DeclarePairedDelimiter{\sqb}{[}{]}
\DeclarePairedDelimiter{\abs}{\lvert}{\rvert}
\DeclarePairedDelimiter{\seminorm}{\lvert}{\rvert}\DeclarePairedDelimiter{\norm}{\lVert}{\rVert}

\DeclarePairedDelimiter{\cycle}{\llbracket}{\rrbracket}
\DeclarePairedDelimiterX{\intvc}[2]{[}{]}{#1,#2}
\DeclarePairedDelimiterX{\intvl}[2]{(}{]}{#1,#2}
\DeclarePairedDelimiterX{\intvr}[2]{[}{)}{#1,#2}
\DeclarePairedDelimiterX{\intvo}[2]{(}{)}{#1,#2}
\newcommand{\VMO}{\mathrm{VMO}}
\newcommand{\BMO}{\mathrm{BMO}}
\newcommand{\energy}{\mathfrak{E}}
\newcommand{\sobolev}{\mathrm{W}}
\newcommand{\Lip}{\mathrm{Lip}}

\newcommand{\lebesgue}{\mathrm{L}}
\newcommand{\continuous}{\mathrm{C}}

\newcommand{\jac}{\mathcal{J}}

\newcommand{\eqpunct}[1]{\,\text{#1}}

\DeclareMathOperator*{\area}{Area}
 
\newcommand\stSymbol[1][]{%
\nonscript\;#1\vert
\allowbreak
\nonscript\;
\mathopen{}}
\DeclarePairedDelimiterX\set[1]\{\}{%
\renewcommand\st{\stSymbol[\delimsize]}
#1
}
\providecommand{\st}{\stSymbol}

\renewcommand{\PrintDOI}[1]{%
  \href{http://dx.doi.org/#1}{doi:#1}%
}

\DefineSimpleKey{bib}{arxiv}
\BibSpec{book}{%
    +{}  {\PrintPrimary}                {transition}
    +{,} { \textit}                     {title}
    +{.} { }                            {part}
    +{:} { \textit}                     {subtitle}
    +{,} { \PrintEdition}               {edition}
    +{}  { \PrintEditorsB}              {editor}
    +{,} { \PrintTranslatorsC}          {translator}
    +{,} { \PrintContributions}         {contribution}
    +{,} { }                            {series}
    +{,} { \voltext}                    {volume}
    +{,} { }                            {publisher}
    +{,} { }                            {organization}
    +{,} { }                            {address} 
    +{,} { \PrintDateB}                 {date}
    +{,} { }                            {status}
    +{,} { \PrintDOI}                   {doi}
    +{}  { \parenthesize}               {language}
    +{}  { \PrintTranslation}           {translation}
    +{;} { \PrintReprint}               {reprint}
    +{.} { }                            {note}
    +{.} {}                             {transition}
    +{}  {\SentenceSpace \PrintReviews} {review}
}
\newcommand{\PrintarXiv}[1]{%
  \href{https://arxiv.org/abs/#1}{arXiv:#1}%
}
\BibSpec{arxiv}{%
    +{}  {\PrintPrimary}                {transition}
    +{,} { \textit}                     {title}
    +{.} { }                            {part}
    +{:} { \textit}                     {subtitle}
    +{,} { \PrintDateB}                 {date}
    +{,} { \PrintarXiv}                 {arxiv}
    +{}  { \parenthesize}               {language}
    +{.} { }                            {note}
    +{.} {}                             {transition}
}

\title%
{%
Heterotopic energy for Sobolev mappings%
}

\author{Antoine Detaille}
\address[A. Detaille]{
Universite Claude Bernard Lyon 1\\ CNRS\\ Centrale Lyon\\ INSA Lyon\\ Université Jean Monnet\\ ICJ UMR5208\\
69622 Villeurbanne\\
France}
\curraddr[A. Detaille]{
ETH Z\"urich\\
Departement Mathematik\\
R\"amistrasse 101\\
8092 Z\"urich \\
Switzerland
}
\email{antoine.detaille@math.ethz.ch }
\author{Jean Van Schaftingen}

\address[J. Van Schaftingen]{
Universit\'e catholique de Louvain, Institut de Recherche en Math\'ematique et Physique, Chemin du Cyclotron 2 bte L7.01.01, 1348 Louvain-la-Neuve, Belgium}
\dedicatory{To the blessed memory of Haïm Brezis,
creative, generous and radiant master of mathematics}
\email{Jean.VanSchaftingen@UCLouvain.be}

\thanks{J. Van Schaftingen was supported by the Projet de Recherche T.0229.21 ``Singular Harmonic Maps and Asymptotics of Ginzburg--Landau Relaxations'' of the Fonds de la Recherche Scientifique--FNRS}

\keywords{}
\subjclass[2020]{58D15 (46E35, 46T10, 58C25)}

\begin{document}

\begin{abstract}
We study the notion of heterotopic energy defined as the limit of Sobolev energies of Sobolev mappings in a given homotopy class approximating almost everywhere a given Sobolev mapping.
We show that the heterotopic energy is finite if and only if the mappings in the corresponding homotopy classes are homotopic on a codimension one skeleton of a triangulation of the domain. When this is the case, the heterotopic energy of a mapping is the sum of its Sobolev energy and its disparity energy, defined as the minimum energy of a bubble to pass between these homotopy classes. 
At the more technical level, we rely on a framework that works when the target and domain manifolds are not simply connected and there is no canonical isomorphism between homotopy groups with different basepoints.
\end{abstract}
\maketitle

\section{Introduction}

Given compact Riemannian manifolds \(\manifold{M}\) and \(\manifold{N}\) (without boundary) with \(m = \dim \manifold{M} \ge 2\) and maps \(u, v \in \continuous^\infty \brk{\manifold{M}, \manifold{N}}\), we are interested in the \emph{heterotopic energy} defined as 
\begin{equation*}
  \energy^{1, m}_{\mathrm{het}}
  \brk{u, v}
  \defeq
  \inf
  \set[\bigg]{
  \liminf_{j \to \infty}
  \int_{\manifold{M}} \abs{\Deriv v_j}^m
  \st
  v_j \in \continuous^\infty \brk{\manifold{M}, \manifold{N}} \text{ is homotopic to } v \text{ and }
  v_j \to u
  }\eqpunct{.}
\end{equation*}
In other words, the heterotopic energy quantifies the cost of approximating a given map with mappings from a fixed homotopy class.
Obviously, this quantity will be mostly interesting only when \(u\) and \(v\) are not homotopic.
In this introduction, we restrict to smooth maps in our definitions and statements for the sake of simplicity, especially when speaking about homotopy classes;
the definition and the discussion extend to a lower regularity setting, as we will carefully discuss in the body of the text starting from \cref{definition_heter_energy}.

Even though, to the best of our knowledge, the notion of heterotopic energy has never been defined as such in the existing literature, such a problem of finding the minimal energy cost for approximating a map from another homotopy class is a pervasive question in the realm of mappings into manifolds.
This is our main motivation for introducing this quantity and studying its main properties.
To be more specific, we list below a few possible applications, along with references. 

First, sequences of mappings converging weakly to a mapping in another homotopy class appear in bubbling phenomena for harmonic and \(p\)-harmonic maps \citelist{\cite{Sacks_Uhlenbeck_1981}\cite{Parker_1996}\cite{Struwe_1985}\cite{Brezis_Coron_1983}\cite{Brezis_Coron_1984}\cite{Lions_1985}*{\S 4.5}}.

Such sequences also appear naturally in the problem of weak approximation of Sobolev mappings in \(\sobolev^{1, p} \brk{\manifold{M}, \manifold{N}}\) where \(p \in \Nset\) and \(p < m = \dim \manifold{M}\), where weakly approximating sequences yield thanks to a Fubini argument and Fatou's lemma sequences converging weakly on \(p\)-dimensional submanifolds and subskeletons \citelist{\cite{Bethuel_2020}\cite{Hardt_Riviere_2003}\cite{Hardt_Riviere_2008}\cite{Pakzad_Riviere_2003}}.
The definition of heterotopic energy shares many features of Bethuel, Brezis, and Coron’s relaxed energy \cite{Bethuel_Brezis_Coron_1990}.

Finally, this quantity provides a way of measuring the distance between homotopy classes, in the spirit of, notably, \citelist{\cite{Rubinstein_Shafrir_2007}\cite{Levi_Shafrir_2014}\cite{Brezis_Mironescu_Shafrir_2016_CRAS}\cite{Brezis_Mironescu_Shafrir_2016}\cite{Shafrir_2018}}, although with no apparent formal mathematical connection. 

The goal of this work is to characterize the heterotopic energy.
To make more natural the definitions that are necessary to state our main result, let us first present a naive strategy to obtain competitors in the definition of \( \energy^{1, m}_{\mathrm{het}}\brk{u, v} \) in the special case where \( \manifold{M} = \manifold{N} = \Sset^{m} \).
We assume for the sake of simplicity \( u \colon \Sset^m \to \Sset^m \) to be constant in some small geodesic ball \( \Bar{B}_{\rho}\brk{a} \subset \Sset^{m} \) with \(\rho\) sufficiently small;
it is possible to return to this situation thanks to the \emph{opening} procedure, that will be explained in~\cref{lemma_opening}.

Since homotopy classes of mappings from \( \Sset^{m} \) to \( \Sset^{m} \) are completely characterized by the degree of Brouwer, we can define \( v_{j}  \colon \Sset^m \to \Sset^m \) as the map obtained from \( u \) by inserting in the smaller ball \( \Bar{B}_{r_{j}}\brk{a} \), with \( 0 < r_{j} < \rho \) and \( r_{j} \to 0 \) as \( j \to \infty \), a map having degree equal to \( \deg{v} - \deg{u} \); with this definition, it is clear that \( v_{j} \) is homotopic to \( v \) and that \( v_{j} \to u \) as \( j \to \infty \).
Hence, \( v_{j} \) is a competitor for \( \smash{\energy^{1, m}_{\mathrm{het}}\brk{u, v}} \).
To obtain a competitor as good as possible, we are led to chose the map that we insert in the ball \( \smash{\Bar{B}_{r_{j}}\brk{a}} \) with an energy as small as possible.
In other words, we aim at taking a map that minimizes the \( \sobolev^{1,m} \) energy among all maps having degree equal to \( \deg{v} - \deg{u} \) and with fixed value on the boundary of the ball.

This seemingly naive strategy is at the core of our work.
To implement it in greater generality, we define, if \(u = w\) in \(\manifold{M} \setminus B_\rho\brk{a}\) with \(\rho\) sufficiently small, the \emph{topological disparity}
\(
  \sqb{u, w, B_\rho\brk{a}}
\) as the homotopy class of maps in \(\continuous^\infty \brk{\Sset^m, \manifold{N}}\) that are homotopic to a map given by \( \smash{u \restr{\Bar{B}_\rho\brk{a}}} \) on the northern hemisphere and by \( \smash{w \restr{\Bar{B}_\rho\brk{a}}} \) on the southern one (see §\ref{section_topological_disparity}).
The \emph{topological energy} of the disparity is then defined as (see §\ref{section_topological_energy})
\[
 \mathfrak{E}^{1, m}_{\mathrm{top}} \brk{ \sqb{u, w, B_\rho\brk{a}}}
 \defeq  
 \inf \set[\bigg]{\int_{\Sset^m} \abs{\Deriv f}^m \st f \in \sqb{u, w, B_\rho\brk{a}} \subseteq \continuous^{\infty} \brk{\Sset^m, \manifold{N}}}\eqpunct{.}
\]
Here and in what follows, we assume that we have fixed once for all an identification of the ball \( \overline{\Bset^m}\) with both  the northern and the southern hemispheres of \(\Sset^m\) that coincide on the equator.
Although the homotopy class \( \sqb{u, w, B_\rho\brk{a}}\) does depend on the choices of orientations when identifying the geodesic ball \( \Bar{B}_{\rho}\brk{a} \subset \Sset^m \) with \( \overline{\Bset^{m}} \), this is not the case for its energy \( \smash{\mathfrak{E}^{1, m}_{\mathrm{top}} \brk{ \sqb{u, w, B_\rho\brk{a}}}} \).
Indeed, choosing a different orientation for the identification of \( \smash{\Bar{B}_{\rho}\brk{a}} \) with \( \overline{\Bset^{m}} \) will result in reflecting the homotopy class \( \sqb{u, w, B_\rho\brk{a}}\) with respect to a hyperplane, which may modify the class, but not its topological energy, which is clearly invariant under isometries.
We define the \emph{disparity energy} of \(u \colon \manifold{M} \to \manifold{N}\) with respect to \(v\colon \manifold{M} \to \manifold{N}\) as (see §\ref{section_disparity_energy})
\begin{multline*}
\energy^{1, m}_{\mathrm{disp}} \brk{u, v}
 \defeq 
 \inf \Bigl\{\energy^{1, m}_{\mathrm{top}} \brk{\sqb{u, w, B_{\rho} \brk{a}}}  \stSymbol[\Big] w \in \continuous^\infty \brk{\manifold{M}, \manifold{N}} \text{ homotopic to \(v\) }\\[-.5em]
 \text{ and } u = w \text{ in \(\manifold{M} \setminus B_\rho\brk{a}\)}\Bigr\}\eqpunct{.}
\end{multline*}
The infimum above runs over all balls \( B_{\rho} \brk{a} \subset \manifold{M} \) with \( \rho \) sufficiently small, and all corresponding maps \( w \).
We will show that the disparity energy induces a distance on homotopy classes in \(\continuous^\infty \brk{\manifold{M}, \manifold{N}}\) (see \cref{proposition_disparity_distance} below).

Our first main result is that the energy disparity is essentially equivalent to the heterotopic energy.
\begin{theorem}
\label{theorem_intro_heterotopic_disparity}
For every \(u, v \in \continuous^\infty \brk{\manifold{M}, \manifold{N}}\),
\begin{equation}
\label{eq_oox5LeeVahZ0theevah1shai}
  \energy^{1, m}_{\mathrm{het}}
  \brk{u, v}
  = \int_{\manifold{M}} \abs{\Deriv u}^m
  + \energy^{1, m}_{\mathrm{disp}} \brk{u, v}\eqpunct{.}
\end{equation}
In particular, \(\energy^{1, m}_{\mathrm{het}}
  \brk{u, v} < \infty\) if and only if \(u\) and \(v\) are homotopic on an \(\brk{m-1}\)-dimensional triangulation of \(\manifold{M}\).
\end{theorem}

The necessity of the homotopy condition for  \( \smash{\energy^{1, m}_{\mathrm{het}}}
  \brk{u, v} < \infty\) in~\cref{theorem_intro_heterotopic_disparity}
is due
to Hang and Lin \cite{Hang_Lin_2003_III}*{Theorem 6.1} (see also \citelist{\cite{White_1988}*{Theorem 2.1}\cite{Hang_Lin_2003_II}*{Theorem 4.1}}).
This condition is part of the pervasive phenomenon of homotopic stability on lower-dimensional sets \citelist{\cite{White_1988}\cite{Brezis_Li_Mironescu_Nirenberg_1999}\cite{Rubinstein_Sternberg_1996}\cite{Hang_Lin_2003_II}}.

When \(\manifold{M} = \Sset^m\) or more generaly if \(\operatorname{id}_{\manifold{M}^{m - 1}}\) is homotopic to a constant in \(\continuous \brk{\manifold{M}^{m - 1}, \manifold{M}}\), then \( \smash{\energy^{1, m}_{\mathrm{het}}}
  \brk{u, v} < \infty\).
On the other hand, we will have \( \smash{\energy^{1, m}_{\mathrm{het}}
  \brk{u, v}} =\infty\) if, given \(u', v' \in \continuous \brk{\manifold{M}', \manifold{N}}\) that are not homotopic with \(\dim \manifold{M}' \le m - 1\) and \(\manifold{M} = \manifold{M}'\times \manifold{M}''\), we consider
  \(u\brk{x',x''} \defeq u' \brk{x'}\) and \(v \brk{x', x''} \defeq v' \brk{x'}\);
  this may happen for instance if \(\pi_{\ell} \brk{\manifold{N}}\not \simeq \set{0}\) with \(\manifold{M}' = \Sset^\ell\), and thus in particular if \(\manifold{N} = \Sset^\ell\): assume that an \( \brk{m-1} \)-dimensional triangulation \( \manifold{M}^{m-1} \) is chosen so that it contains a product of \( \manifold{M}' \) with a triangulation of \( \manifold{M}'' \) -- we note that being homotopic on a triangulation does not depend on the choice of the triangulation;
  in this case, \( u \) and \( v \) cannot be homotopic on \( \manifold{M}^{m-1} \), for otherwise it would restrict to a homotopy on \( \manifold{M}' \).
  We will also have \(\smash{\energy^{1, m}_{\mathrm{het}}
  \brk{u, v}} =\infty\) if \( u \) is constant and \( v \in \continuous^\infty \brk{\RPset^m, \RPset^n}\) is the canonical injection of \(\RPset^m\) in \(\RPset^n\) with \(m \le n\), since \(v \restr{\RPset^1}\) is then the generator of \(\pi_1 \brk{\RPset^n}\).

  \medskip

The identity \eqref{eq_oox5LeeVahZ0theevah1shai} in~\cref{theorem_intro_heterotopic_disparity} states that one \emph{cannot do better} than the naive strategy that we have just presented: modifying the map \( u \) only in a small ball to obtain maps homotopic to \( v \) and converging to \( u \) yields the optimal contribution to \( \energy^{1, m}_{\mathrm{het}}\brk{u, v} \).

In some special situations, the expression of the heterotopic energy can be simplified, and allows to recover some familiar formulas.
\Cref{corollary_intro_spheres}, \cref{corollary_intro_projective}, and \cref{corollary_intro_area} below follow directly from \cref{theorem_intro_heterotopic_disparity}, the definition of the disparity energy, and \cref{proposition_Etop_sphere}, \cref{proposition_Etop_projective}, or \cref{proposition_Etop_area} respectively.

A model example of \cref{theorem_intro_heterotopic_disparity}  deals with mappings into spheres, where the heterotopic energy can be computed from the difference of degree between the maps \( u \) and \( v \).

\begin{corollary}[Mappings into spheres]
\label{corollary_intro_spheres}
	For every \( u, v \in \continuous^{\infty}\brk{\manifold{M}, \Sset^{m}} \),
	\begin{enumerate}[label=(\roman*)]
		\item if \( \manifold{M} \) is orientable, then 
		\[
			\energy^{1, m}_{\mathrm{het}}
			\brk{u, v}
			= \int_{\manifold{M}} \abs{\Deriv u}^m
			+ m^{m/2}\abs{\Sset^{m}}\abs{\deg{u}-\deg{v}}\eqpunct{,}
		\]
		\item if \( \manifold{M} \) is not orientable, then
		\[
			\energy^{1, m}_{\mathrm{het}}
			\brk{u, v}
			=
			\left\{
			\begin{aligned}
				&\int_{\manifold{M}} \abs{\Deriv u}^m && \text{if \( u \) and \( v \) are homotopic,}\\
				&\int_{\manifold{M}} \abs{\Deriv u}^m
				+ m^{m/2}\abs{\Sset^{m}} & &\text{otherwise.}
			\end{aligned}
			\right.
		\]
	\end{enumerate}
\end{corollary}

At the heart of the above formula is the fact that, for mappings between spheres, the degree of Brouwer \emph{completely encodes} homotopy classes.
The difference from the orientable and the non-orientable case comes from the fact that the degree of Brouwer is well-defined for sphere-valued maps defined on an orientable manifold, while when the domain is not orientable, there are only two homotopy classes of maps into \( \Sset^{m} \) (see also the remark after \cref{lemma_disparity_trivial_homotopic}).

A similar situation occurs for mappings into the real projective spaces \( \RPset^m\).

\begin{corollary}[Mappings into projective spaces]
\label{corollary_intro_projective}
	Assume that \( \manifold{M} \) is simply connected.
	For every \( u, v \in \continuous^{\infty}\brk{\manifold{M}, \RPset^{m}} \),
	\[
		\energy^{1, m}_{\mathrm{het}}
		\brk{u, v}
		= \int_{\manifold{M}} \abs{\Deriv u}^m
		+  m^{m/2}\,2\abs{\RPset^{m}}d_{u,v}\eqpunct{.}
	\]
\end{corollary}

In \cref{corollary_intro_projective}, \( d_{u,v} \) is defined in terms of the covering \(\lifting{u}\) and \(\lifting{v} \in \continuous \brk{\manifold{M}, \Sset^m}\) of \(u\) and \(v\) respectively, as \(d_{u,v} \defeq \abs{\deg \lifting{u} - \deg \lifting{v}} \) if \(m\) is odd and \( d_{u,v} \defeq  \abs{\abs{\deg \lifting{u}} - \abs{\deg \lifting{v}}}\) if it is even.

In \cref{corollary_intro_spheres} and \cref{corollary_intro_projective}, the energy gap between the heterotopic energy and the Sobolev energy is linear with respect to the difference of degree between the maps \( u \) and \( v \).
As expressed by \cref{theorem_intro_heterotopic_disparity}, this is directly related to the rate of growth of the minimal energy required to construct a map with fixed degree.
For a general target manifold \( \manifold{N} \) however, the such rate of growth \emph{need not} be linear.
For more details as well as references, see~\cref{proposition_Etop_Hopf} and the comment below the proposition.

A common feature in both \cref{corollary_intro_spheres} and \cref{corollary_intro_projective} is that the energy involves the \emph{area} of the homotopy class formed by the difference between \( u \) and \( v \).
Our next statement expresses that this is not an isolated phenomenon, and that there is always a relation between the heterotopic energy and the minimal area enclosed by the maps \( u \) and \( v \), provided that one works with the language of \emph{homology}.

\begin{corollary}
\label{corollary_intro_area}
	Assume that \( \manifold{M} \) is orientable and that \( m \leq \dim{\manifold{N}} \).
	For every \( u, v \in \continuous^{\infty} \brk{\manifold{M}, \manifold{N}} \),
	\[
		\energy^{1, m}_{\mathrm{het}}
		\brk{u, v}
		\geq \int_{\manifold{M}} \abs{\Deriv u}^m
		+ m^{m/2}\area\brk{\sqb{u\brk{\manifold{M}}}-\sqb{v\brk{\manifold{M}}}}\eqpunct{,}
	\]
	and equality occurs when \( m =  2 \) and the Hurewicz homomorphism \( \mathfrak{hur} \colon \pi_{m}\brk{\manifold{N}} \to H_{m}\brk{\manifold{N}} \) is an isomorphism,
	where \( \area\brk{\sqb{u\brk{\manifold{M}}}-\sqb{v\brk{\manifold{M}}}} \) denotes the minimal area of a map realizing the homology class \( \sqb{u\brk{\manifold{M}}}-\sqb{v\brk{\manifold{M}}} \), taking into account the multiplicity.
\end{corollary}

As will be shown, the lower estimate follows from the arithmetico-geometric inequality to relate the Jacobian and the differential of a map, while the equality case follows from the \emph{Morrey \( \varepsilon \)-conformality theorem}~\cite{Morrey_1948}, which is available only in dimension \( 2 \).

The assumption that \( \manifold{M} \) is orientable ensures that \( \manifold{M} \) itself is a homology cycle, so that the cycles \( \sqb{u\brk{\manifold{M}}} \) and \( \sqb{v\brk{\manifold{M}}} \) are well-defined.
In this case, for any ball \( B_{\rho}\brk{a} \subset \manifold{M} \), the homology class associated with the disparity between \( u \) and a map \( w \) homotopic to \( v \) may be computed as
\[
	\mathfrak{hur}\brk{\sqb{u,w,B_{\rho}\brk{a}}} = \sqb{u\brk{\manifold{M}}}-\sqb{w\brk{\manifold{M}}}
	=
	\sqb{u\brk{\manifold{M}}}-\sqb{v\brk{\manifold{M}}}\eqpunct{.}
\]

The increasing complexity of the statements of the above corollaries pertains to a crucial technical difficulty when working with free homotopy classes, and that will be of key importance to us in our choice of formalism to work with.
When \(\manifold{M} = \Sset^m\), the homotopy class \(\sqb{u, w, B_{\rho} \brk{a}}\) appearing in the definition of the disparity energy \( \smash{\energy^{1, m}_{\mathrm{disp}}} \brk{u, v}\) is the difference in \(\pi_{m} \brk{\manifold{N}, u \brk{c}}\) between the classes corresponding to \(u\) and \(w\) for any \(c \in \manifold{M} \setminus B_{\rho}\brk{a}\).
Letting \(b = u \brk{c}\), the class corresponding to \(w\) will be independent on \(b\) if and only if the the action of \(\pi_1 \brk{\manifold{N}, b}\) on \(\pi_{m}\brk{\manifold{N}, b}\) is trivial. 
This will be the case for instance when \(\pi_1 \brk{\manifold{N}, b} \simeq \set{0}\), including in particular the case \(\manifold{N} = \Sset^m\) (see \cref{corollary_intro_spheres}) or when \(\pi_1 \brk{\manifold{N}, b}\) acts trivially on \(\pi_{m} \brk{\manifold{N},b}\), covering the case \(\manifold{N} = \RPset^m\) with \(m\) odd (see \cref{corollary_intro_projective}).
Otherwise, one needs to take into account the action of \(\pi_1 \brk{\manifold{N}, b}\) on \(\pi_{m}\brk{\manifold{N}, b}\); this is what \cref{corollary_intro_projective} does when \(m\) is even.
On the domain side, the homotopy classes from a general domain \(\manifold{M}\) to \(\manifold{N}\) do not have in general a group structure, and we need thus a formalism that gives a meaning to the difference between homotopy classes. 
These considerations show that purely algebraic manipulations in \(\pi_{m}\brk{\manifold{N}}\) cannot describe the difference in topology and justify our need to rely on a formalism and computation methods that take into account the way individual topological charges are interlinked.

\bigskip
At the core of the proof of the upper estimate on \( \energy^{1, m}_{\mathrm{het}}\brk{u, v} \) lies the insertion strategy that we sketched above.
Concerning the lower estimate, it relies on an instance of a \emph{bubbling phenomenon}.
Therefore, we next present a bubbling statement fitted to our purposes.

\begin{theorem}
\label{theorem_intro_heterotopic_concentration}
For every \(u, v \in \continuous^\infty \brk{\manifold{M}, \manifold{N}}\),
if \(\brk{v_j}_{j \in \Nset}\) is a sequence in \(\continuous^\infty \brk{\manifold{M}, \manifold{N}}\) that converges almost everywhere to \(u\), if for every \(j \in \Nset\), \(v_j\) is homotopic to \(v\), and if
\[
 \sup_{j \in \Nset}\int_{\manifold{M}} \abs{\Deriv v_j}^m <\infty\eqpunct{,}
\]
then up to a subsequence, there exist points \(a_1, \dotsc, a_I \in \manifold{M}\) such that, for every radius \( \rho \in \intvo{0}{\infty} \) sufficiently small, there exists a map \(w \in \continuous^\infty \brk{\manifold{M}, \manifold{N}}\) such that 
\(w = u\) in \(\manifold{M} \setminus \bigcup_{i = 1}^I B_\rho\brk{a_i}\), \(w\) is homotopic to \(v\), for every \(i \in \set{1, \dotsc, I}\),
\begin{equation}
\label{eq_iasooYu9pohhuf7Jai0Eghoo}
 \lim_{r \to 0} \liminf_{j \to \infty} \int_{B_r \brk{a_i}} \abs{\Deriv v_j}^m
 \ge \mathfrak{E}^{1, m}_{\mathrm{top}} \brk{ \sqb{u,w, B_\rho\brk{a_i}}}\eqpunct{,}
\end{equation}
and 
\begin{equation}
\label{eq_quivu5Dai3Oongis9ohPooch}
 \sum_{i = 1}^I \mathfrak{E}^{1, m}_{\mathrm{top}} \brk{ \sqb{u, w, B_\rho\brk{a_i}}}
 \ge \energy^{1, m}_{\mathrm{disp}} \brk{u, v}\eqpunct{.}
\end{equation}
\end{theorem}

Bubbling statements, such as~\cref{theorem_intro_heterotopic_concentration} and its companion~\cref{theorem_concentration_measures} in the body of the text, are ubiquitous in the study of weakly converging maps; see for instance, but not only,~\citelist{\cite{Giaquinta_Modica_Soucek_1998_II}*{Theorem~3.1.5.1}\cite{Hang_Lin_2003_III}*{Theorem~10.1}\cite{Hardt_Riviere_2008}*{Proposition~3.4}\cite{Bethuel_2020}*{Remark~1}\cite{Detaille_VanSchaftingen}*{Proposition~2.1}}.
We draw the attention of the reader to the very sharp and general character of both bubbling results that we present here: they apply to any weakly converging sequence, and they relate the atoms of the limiting measure towards which the convergence of the differential occurs to the topological defect between the homotopy classes, which can be realized by \emph{one} given map \( w \), homotopic to \( v \), and which differs from \( u \) only on a finite number of small balls.
We believe that stating such precise results and providing them with a complete proof is also of independent interest.

\section{The heterotopic energy on Sobolev spaces}
\subsection{Continuous, VMO, and Sobolev homotopies}

Whereas the introduction 
was restricted to the case of smooth maps for the simplicity of the exposition, our discussion can be extended to a lower regularity framework where the notion of homotopy carries out.
It will therefore be instrumental to our endeavor to consider homotopies between low regularity maps.
For this purpose, we begin this section with a short exposition of the notion of homotopy in various functions spaces.
We start with the most classical setting of continuous maps.

Two maps \(u, v \in \continuous  \brk{\manifold{M}, \manifold{N}}\) are said to be homotopic in \(\continuous\brk{\manifold{M}, \manifold{N}}\) whenever there exists a mapping \(H \in \continuous\brk{\intvc{0}{1}\times \manifold{M}, \manifold{N}}\) such that 
\(H \brk{0, \cdot} = u\) and \(H \brk{1, \cdot}=v\). 
When \(\manifold{M}\) is compact, \(u\) and \(v\) are homotopic if and only if there exists \(H \in \continuous\brk{\intvc{0}{1}, \continuous \brk{\manifold{M}, \manifold{N}}}\) such that \(H \brk{0} = u\) and \(H \brk{1} = v\), where the space \(\continuous \brk{\manifold{M}, \manifold{N}}\) is endowed with the topology of the uniform distance.
Moreover, if \(\manifold{N}\) is compact, there exists \(\delta \in \intvo{0}{\infty}\) such that if  \(d \brk{u, v}\le \delta\) everywhere in \(\manifold{M}\), then \(u\) and \(v\) are homotopic. 
Indeed, it suffices to take \(H \brk{t, x} = \Pi_{\manifold{N}} \brk{\brk{1 - t}u \brk{x}+ t v \brk{x}}\), where \(\Pi_{\manifold{N}}\) is the nearest-point retraction of a neighbourhood of \(\manifold{N}\) on \(\manifold{N}\).

Most of the homotopy theory carries on in the framework of maps of \emph{vanishing mean oscillation} \citelist{\cite{Brezis_Nirenberg_1995}\cite{Brezis_Nirenberg_1996}} (see also \cite{Brezis_1997}).
A function \(u \colon \manifold{M} \to \Rset^\nu\) belongs to \(\VMO \brk{\manifold{M}, \Rset^\nu}\) whenever
\[
  \lim_{\delta \to 0}
  \sup_{\substack{a \in \manifold{M}\\ 0 < r < \delta}}
  \fint_{B_r \brk{a}} \fint_{B_r \brk{a}}
  \abs{u \brk{x} - u \brk{y}} \dif x \dif y
  = 0\eqpunct{.}
\]
The space \( \VMO\brk{\manifold{M}, \Rset^{\nu}} \) is endowed with the norm
\[
  \norm{u}_{\BMO}
  \defeq \norm{u}_{\lebesgue^1} + \seminorm{u}_{\BMO}\eqpunct{,}
\]
where
\[
  \seminorm{u}_{\BMO}
  \defeq
  \sup_{\substack{a \in \manifold{M}\\ r > 0}}
  \fint_{B_r \brk{a}} \fint_{B_r \brk{a}}
  \abs{u \brk{x} - u \brk{y}} \dif x \dif y\eqpunct{.}
\]
Continuous functions are a dense subset of \(\VMO \brk{\manifold{M}, \Rset^\nu}\) \cite{Sarason_1975}.
One defines then
\[
\VMO \brk{\manifold{M}, \manifold{N}}
\defeq
\set[\big]{u \in \VMO \brk{\manifold{M}, \Rset^\nu}
\st u \in \manifold{N} \text{ almost everywhere}}\eqpunct{.}
\]
The topology of \(\VMO \brk{\manifold{M}, \manifold{N}}\) -- and of \(\VMO \brk{\manifold{M}, \Rset^\nu}\) -- can also be described by the basis of open sets (see  \cite{Brezis_Nirenberg_1995}*{Lemma A.16})
\[
 \set[\bigg]{v \colon \manifold{M} \to \manifold{N} \st
 \int_{\manifold{M}} d \brk{u, v} < \varepsilon
 \text{ and }
 \sup_{\substack{a \in \manifold{M}\\ 0 < r < \delta}}
  \fint_{B_r \brk{a}} \fint_{B_r \brk{a}}
  d \brk{u \brk{x}, u \brk{y}} \dif x \dif y
  < \eta
 }\eqpunct{,}
\]
when \( u \) runs over all maps in \(\VMO \brk{\manifold{M}, \manifold{N}} \) and \( \varepsilon \), \( \delta \), and \( \eta \) run over all positive numbers.

One can then define maps \(u\) and \(v \in \VMO \brk{\manifold{M}, \manifold{N}}\) to be homotopic in \(\VMO \brk{\manifold{M}, \manifold{N}}\) whenever there exists a mapping \(H \in \continuous \brk{\intvc{0}{1}, \VMO \brk{\manifold{M}, \manifold{N}}}\) such that \(H \brk{0} = u\) and \(H \brk{1} = v\).

Homotopy classes are also open sets in \(\VMO \brk{\manifold{M}, \manifold{N}}\):

\begin{proposition}[Brezis \& Nirenberg \cite{Brezis_Nirenberg_1995}*{Lemma A.19}]
The path-connected components of \(\VMO \brk{\manifold{M}, \manifold{N}}\) are open.
\end{proposition}

For continuous maps, homotopies in \(\VMO \brk{\manifold{M}, \manifold{N}}\) and in \(\continuous \brk{\manifold{M}, \manifold{N}}\) are equivalent:
\begin{proposition}[Brezis \& Nirenberg \cite{Brezis_Nirenberg_1995}*{Lemma A.20}]
\label{proposition_C_and_VMO_homotopies}
Given \(u, v \in \continuous \brk{\manifold{M}, \manifold{N}}\), the following are equivalent:
\begin{enumerate}[label=(\roman*)]
 \item \(u\) and \(v\) are homotopic in \(\VMO \brk{\manifold{M}, \manifold{N}}\),
 \item \(u\) and \(v\) are homotopic in \(\continuous \brk{\manifold{M}, \manifold{N}}\).
\end{enumerate}
\end{proposition}

Since \(\continuous \brk{\manifold{M}, \manifold{N}}\) is open, it follows that every path-connected component of \(\VMO \brk{\manifold{M}, \manifold{N}}\) contains a unique path-connected component of \(\continuous \brk{\manifold{M}, \manifold{N}}\) \cite{Brezis_Nirenberg_1995}*{Lemma A.21}.
A stronger result shows that the inclusion \(\continuous \brk{\manifold{M}, \manifold{N}} \subseteq \VMO \brk{\manifold{M}, \manifold{N}}\) is a homotopy equivalence \cite{Abbondandolo_1996}.

Even though it will not be used in the sequel, it is conceptually interesting to note that homotopies in \(\VMO \brk{\manifold{M}, \manifold{N}}\) can be characterized with maps in \(\VMO \brk{\intvc{0}{1}\times \manifold{M}, \manifold{N}}\):
\begin{proposition}[Brezis \& Nirenberg \cite{Brezis_Nirenberg_1996}*{Corollary 3}]
The maps \(u, v\in \VMO \brk{\manifold{M},\manifold{N}}\) are homotopic in \(\VMO \brk{\manifold{M},\manifold{N}}\) if and only if there exists a mapping \(H \in \VMO \brk{\intvc{0}{1}\times \manifold{M}, \manifold{N}}\) such that
if \(t \in \intvc{0}{1/3}\), \(H \brk{t, \cdot} = u\) and if \(t \in \intvc{2/3}{1}\), \(H \brk{t, \cdot} = v\).
\end{proposition}

We finally consider Sobolev mappings.
We say that the maps \( u, v \in \sobolev^{1, m}\brk{\manifold{M}, \manifold{N}}\) are homotopic whenever there is some mapping
\(H \in \continuous \brk{\intvc{0}{1}, \sobolev^{1, m}\brk{\manifold{M}, \manifold{N}}}\) such that \(H \brk{0} = u\) and \(H \brk{1} = v\).

\begin{proposition}
\label{proposition_stability_homotopy_classes}
The path-connected components of \(\sobolev^{1, m} \brk{\manifold{M}, \manifold{N}}\) are open.
\end{proposition}

The proof of \cref{proposition_stability_homotopy_classes} relies on Schoen and Uhlenbeck's seminal observation that, even though averages of maps in \( \sobolev^{1, m}\brk{\manifold{M}, \manifold{N}}\) do not converge uniformly, they still take values close to the target manifold \(\manifold{M}\) \cite{Schoen_Uhlenbeck_1983} (see \cite{Hang_Lin_2003_II}*{\S 4} for detailed similar arguments for homotopies).

\Cref{proposition_stability_homotopy_classes} still holds in \( \sobolev^{1, p}\brk{\manifold{M}, \manifold{N}}\) with \(p \ne m\) in an even weaker form: connected components are sequentially weakly closed rather than strongly closed;
when \(p > m\) this is a standard application of the Sobolev--Morrey embedding and Arzelà--Ascoli's compactness criterion; when \(p < m\) this is due to Hang and Lin \cite{Hang_Lin_2003_II}.

The space \( \VMO \) is the largest among the three that we consider in this section: we clearly have \( \continuous\brk{\manifold{M},\manifold{N}} \hookrightarrow \VMO\brk{\manifold{M},\manifold{N}} \), and it also holds that \( \sobolev^{1,m}\brk{\manifold{M},\manifold{N}} \hookrightarrow \VMO\brk{\manifold{M},\manifold{N}} \) by virtue of the limiting case of the Sobolev--Morrey embedding; see, e.g.,~\cite{Brezis_Nirenberg_1995}*{Example~1}.
The notions of homotopy in \(\sobolev^{1, m}\brk{\manifold{M}, \manifold{N}}\), \(\VMO \brk{\manifold{M}, \manifold{N}}\), and \(\continuous \brk{\manifold{M}, \manifold{N}}\) are equivalent \cite{Brezis_Li_2001}; the proof also relies on Schoen and Uhlenbeck's estimate.

\begin{proposition}
\label{proposition_equivalence_homotopies}
Given \(u, v \in \sobolev^{1, m}\brk{\manifold{M}, \manifold{N}}\), the following are equivalent:
\begin{enumerate}[label=(\roman*)]
 \item \(u\) and \(v\) are homotopic in \(\sobolev^{1, m}\brk{\manifold{M}, \manifold{N}}\),
 \item \(u\) and \(v\) are homotopic in \(\VMO \brk{\manifold{M}, \manifold{N}}\).
\end{enumerate}
If moreover \(u\) and \(v\) are continuous, then the previous assertions are equivalent to
\begin{enumerate}[label=(\roman*), resume]
 \item \(u\) and \(v\) are homotopic in \(\continuous \brk{\manifold{M}, \manifold{N}}\).
\end{enumerate}
\end{proposition}

For both continuous and \( \VMO \) maps, homotopies could equivalently be defined as mappings on the product \( \intvc{0}{1} \times \manifold{M} \) with suitable boundary condition, or continuous mappings from \( \intvc{0}{1} \) to the corresponding space on \( \manifold{M} \).
In Sobolev spaces, a suitable notion of restriction to the boundary is given by the trace.
However, in contrast to the situation in \(\continuous \brk{\manifold{M}, \manifold{N}}\) and \(\VMO \brk{\manifold{M}, \manifold{N}}\),
one can have \(u = \smash{\tr_{\set{0}\times \manifold{M}}} H\) and \(v = \smash{\tr_{\set{1}\times \manifold{M}}} H\) with \(H \in \sobolev^{1, m}\brk{\intvc{0}{1}\times \manifold{M}, \manifold{N}}\) without having \(u\) and \(v\) homotopic in \(\sobolev^{1, m}\brk{\manifold{M}, \manifold{N}}\). Indeed, one can take \(u, v \in \continuous^\infty \brk{\manifold{M}, \manifold{N}}\) such that \(u = v\) on an \(\brk{m- 1}\)-dimensional triangulation \(\manifold{M}^{m - 1}\) and construct \(H\) by homogeneous extension from \(\intvc{0}{1}\times \manifold{M}^{m - 1} \cup \set{0, 1}\times \manifold{M}\).

\subsection{Heterotopic energy}

With these reminders about homotopies, we are now in position to define the heterotopic energy for maps of lower regularity.

\begin{definition}
\label{definition_heter_energy}
Given \(u \colon \manifold{M} \to \manifold{N}\) and \(v \in \VMO\brk{\manifold{M}, \manifold{N}}\), we define the heterotopic energy of \(u\) with respect to \(v\) as 
\begin{multline}
\label{eq_feTaev4ahyi4aeTheem8Uiku}
  \energy^{1, m}_{\mathrm{het}}
  \brk{u, v}
  \defeq
  \inf
  \biggl\{
  \liminf_{j \to \infty}
  \int_{\manifold{M}} \abs{\Deriv v_j}^m
  \stSymbol[\bigg]
  v_j \in \sobolev^{1, m} \brk{\manifold{M}, \manifold{N}} \text{ is homotopic to } v \\[-1em]
  \text{ and }
  v_j \to u \text{ a.e.\ in \(\manifold{M}\)}
  \biggr\}\eqpunct{.}
\end{multline}
\end{definition}
Equivalently, one has 
\begin{multline}
\label{eq_800de20c207671d6}
  \energy^{1, m}_{\mathrm{het}} \brk{u, v}
  = 
  \sup_{\delta > 0}
  \inf \biggl\{
  \int_{\manifold{M}}\!\! \abs{\Deriv w}^m 
  \stSymbol[\bigg] w \in \sobolev^{1, m} \brk{\manifold{M}, \manifold{N}} \text{ is homotopic to } v \\[-1em]
  \text{ and } \int_{\manifold{M}} d \brk{u, w}^m \le \delta\biggr\}\eqpunct{.}
\end{multline}

The following proposition shows that we can further restrict the infimum in~\eqref{eq_800de20c207671d6} to smooth maps.
\begin{proposition}
\label{proposition_heterotopic_restricts_to_smooth}
If \(u \colon \manifold{M} \to \manifold{N}\) and \(v \in \VMO\brk{\manifold{M}, \manifold{N}}\), then
\begin{multline*}
  \energy^{1, m}_{\mathrm{het}}\brk{u, v}
  = 
  \sup_{\delta > 0}
  \inf \biggl\{\int_{\manifold{M}} \abs{\Deriv w}^m 
  \stSymbol[\bigg] w \in \continuous^\infty \brk{\manifold{M}, \manifold{N}} \text{ is homotopic to } v \\[-1em]
  \text{ and } \int_{\manifold{M}} d \brk{u, w}^m \le \delta\biggr\}\eqpunct{.}
\end{multline*}
\end{proposition}
\begin{proof}
This readily follows from \eqref{eq_800de20c207671d6}, the density of smooth maps in \(\sobolev^{1, m} \brk{\manifold{M}, \manifold{N}}\)~\cite{Schoen_Uhlenbeck_1983}, and the stability of homotopy classes in \(\sobolev^{1, m} \brk{\manifold{M}, \manifold{N}}\) (\cref{proposition_stability_homotopy_classes}).
\end{proof}

We next state a straightforward lower bound on the heterotopic energy, given by the Sobolev energy of the map itself.
In particular, it implies that a map \( u \colon \manifold{M} \to \manifold{N} \) can have a finite heterotopic energy only if it belongs to \( \sobolev^{1,m} \brk{\manifold{M}, \manifold{N}} \).
\begin{proposition}
\label{proposition_heter_energy_controls_energy}
For every \(u \colon \manifold{M} \to \manifold{N}\) and \(v \in \VMO\brk{\manifold{M}, \manifold{N}}\),
\[
\int_{\manifold{M}} \abs{\Deriv u}^m
\le 
 \energy^{1, m}_{\mathrm{het}}
  \brk{u, v}\eqpunct{.}
\]
In particular, if \( \energy^{1, m}_{\mathrm{het}}
  \brk{u, v} < \infty\), then \(u \in \sobolev^{1, m} \brk{\manifold{M}, \manifold{N}}\subseteq \mathrm{VMO} \brk{\manifold{M}, \manifold{N}}\).
\end{proposition}
\begin{proof}
If \(\energy^{1, m}_{\mathrm{het}}
  \brk{u, v} < \infty \), then by definition in \eqref{eq_feTaev4ahyi4aeTheem8Uiku} there exists a sequence \(\brk{v_j}_{j \in \Nset}\) in \( \sobolev^{1,m}\brk{\manifold{M,\manifold{N}}} \) such that 
  \[
    \limsup_{j \to \infty} \int_{\manifold{M}} \abs{\Deriv v_j}^m
    \le 
    \energy^{1, m}_{\mathrm{het}}
  \brk{u, v}
  \]
  and \(v_j \to u\) almost everywhere in \(\manifold{M}\); 
  by lower semicontinuity we then get that \(u\) is weakly differentiable and 
  \[
   \int_{\manifold{M}} \abs{\Deriv u}^m \le  \liminf_{j \to \infty} \int_{\manifold{M}} \abs{\Deriv v_j}^m\eqpunct{.}\qedhere
  \]
\end{proof}

The following proposition is nothing else but a lower semicontinuity property of the heterotopic energy with respect to the convergence in measure.
\begin{proposition}
\label{proposition_Ehet_lsc}
If \(\brk{u_j}_{j \in \Nset}\) is a sequence of measurable mappings from \(\manifold{M}\) to \(\manifold{N}\) converging to  \(u \colon \manifold{M}\to \manifold{N}\) in measure, then 
\[
 \energy^{1, m}_{\mathrm{het}}
  \brk{u, v}
  \le \liminf_{j \to \infty} \energy^{1, m}_{\mathrm{het}}
  \brk{u_j, v}\eqpunct{.}
\]
\end{proposition}
\begin{proof}
This follows from the definition of heterotopic energy \( \energy^{1, m}_{\mathrm{het}} \brk{u_j, v}\) (\cref{definition_heter_energy}) and a diagonal argument.
\end{proof}

We conclude this section with a straightforward exact computation of the heterotopic energy of a map with respect to another map in the same homotopy class.
\begin{proposition}
If \(u \in \VMO \brk{\manifold{M}, \manifold{N}}\) and \(v \in \VMO\brk{\manifold{M}, \manifold{N}}\) is homotopic to \(u\) in \( \VMO \brk{\manifold{M}, \manifold{N}} \), then 
\begin{equation}
\label{eq_aiwuuMael2chaiya7uloh1mi}
 \energy^{1, m}_{\mathrm{het}}
  \brk{u, v}
  = \int_{\manifold{M}} \abs{\Deriv u}^m\eqpunct{.}
\end{equation}
\end{proposition}
\begin{proof}
The lower bound on \( \energy^{1, m}_{\mathrm{het}}\brk{u, v} \) follows directly from~\cref{proposition_heter_energy_controls_energy}.
Moreover, in view of \cref{proposition_heter_energy_controls_energy}, we can assume that \(u \in \sobolev^{1, m}\brk{\manifold{M}, \manifold{N}}\), or equivalently, that the right-hand side of \eqref{eq_aiwuuMael2chaiya7uloh1mi} is finite.
Therefore, we can use \( u \) itself as a competitor in the equivalent definition~\eqref{eq_800de20c207671d6}, which immediately gives the upper bound and concludes the proof.
\end{proof}

\section{Finiteness criterion}

The goal of this section is to prove that the heterotopic energy can only be finite for mappings that are homotopic on a suitable skeleton.

We assume that we are given a \emph{triangulation} of our domain manifold \(\manifold{M}\) once for all, that is, we have a finite \(m\)-dimensional simplicial complex and a homeomorphism between the realisation of this complex and \( \manifold{M} \) whose restriction to any closed simplex of the complex is a smooth diffeomorphism on its image in \(\manifold{M}\).
Since our domain manifold \( \manifold{M} \) is smooth, such a triangulation always exists~\cite{Cairns_1935}.
For every \(\ell \in \set{0, \dotsc, m}\), we let \(\manifold{M}^\ell\) denote the \(\ell\)-dimensional component of \(\manifold{M}\) defined as the union of the images of the closed \(\ell\)-dimensional simplices of the simplicial complex defining our triangulation.

\begin{theorem}
\label{theorem_Ehet_finite}
Let \(u \colon \manifold{M} \to \manifold{N}\)
and let \(v \in \continuous \brk{\manifold{M}, \manifold{N}}\).
The following are equivalent:
\begin{enumerate}[label=(\roman*)]
\item 
\label{it_eKai9ahmoes3fe1biquah2pe}
\(
 \energy^{1, m}_{\mathrm{het}}
  \brk{u, v} < \infty\),
\item 
\label{it_eiPh1baisahth4zieFee8xoh}
\(u \in \sobolev^{1, m}\brk{\manifold{M}, \manifold{N}}\) and 
\(u\) is homotopic in \(\VMO \brk{\manifold{M}, \manifold{N}}\) to some \(w \in \continuous \brk{\manifold{M}, \manifold{N}}\) such that 
\(w \restr{\manifold{M}^{m - 1}} = v \restr{\manifold{M}^{m - 1}}\),
\item 
\label{it_aitoo2kePh0Ue2Teitaechae}
\(u \in \sobolev^{1, m}\brk{\manifold{M}, \manifold{N}}\) and
\(\tr_{\manifold{M}^{m- 1}} u\) is homotopic in \(\VMO \brk{\manifold{M}^{m - 1}, \manifold{N}}\) to \(v \restr{\manifold{M}^{m - 1}}\).
\end{enumerate}
\end{theorem}

The space \(\VMO \brk{\manifold{M}^{m - 1}, \manifold{N}}\) is described similarly to \(\VMO \brk{\manifold{M}, \manifold{N}}\), using the measure by the Riemannian metric on \(\manifold{M}^{m - 1}\) and the intersections of geodesics balls with \(\manifold{M}^{m - 1}\).
Similarly, one can define \emph{fractional} Sobolev spaces \( \sobolev^{s,p} \brk{\manifold{M}^{m-1},\manifold{N}} \), with \( 0 < s < 1 \), using the definition via the Gagliardo seminorm relying on the Riemannian metric on \(\manifold{M}^{m - 1} \).
By the fractional limiting case of the Sobolev--Morrey embedding combined with Gagliardo's trace theorem, if \( u \in \sobolev^{1, m} \brk{\manifold{M}, \manifold{N}} \), then \( \tr_{\manifold{M}^{m  - 1}} u \in \VMO \brk{\manifold{M}^{m - 1}, \manifold{N}} \) and the corresponding mapping is continuous.
Let us note however that we are not defining integer order Sobolev spaces on \( \manifold{M}^{m-1} \).

The remaining part of this section is devoted to the proof of~\cref{theorem_Ehet_finite}.
Obtaining~\ref{it_aitoo2kePh0Ue2Teitaechae} from~\ref{it_eiPh1baisahth4zieFee8xoh} is readily done by a standard argument involving the continuity of traces.
To go from~\ref{it_eKai9ahmoes3fe1biquah2pe} to~\ref{it_eiPh1baisahth4zieFee8xoh}, the core of the argument is a standard homotopy result for maps that are sufficiently close in \( \lebesgue^{m} \) and have sufficiently small \( \sobolev^{1,m} \) energy.
Such a result is stated in~\cref{proposition_homotopy_skeleton_quantitative}, which itself relies on~\cref{lemma_uniform_estimate_skeleton}, and is in line with Schoen and Uhlenbeck's seminal estimate.

In order to obtain~\ref{it_eKai9ahmoes3fe1biquah2pe} from~\ref{it_aitoo2kePh0Ue2Teitaechae}, we proceed in two steps.
We first explain, when \( \tr_{\manifold{M}^{m  - 1}} u \) is homotopic to \( v \restr{\manifold{M}^{m - 1}} \), how to replace \( v \) with another map, with the same trace as \( u \) on \( \manifold{M}^{m - 1} \), and which is homotopic to \( v \) on the whole \( \manifold{M} \).
This task is carried out by~\cref{lemma_equalising_on_M_m_1} and~\cref{lemma_homotopy}, and relies on a cylinder insertion construction, to suitably modify the values of \( v \) near the skeleton \( \manifold{M}^{m - 1} \).
Thanks to this first step, one may assume that \( u \) and \( v \) have the same trace on \( \manifold{M}^{m - 1} \).
The second step, which is contained in~\cref{lemma_bubble_from_2_disks}, consists of inserting, on a small ball inside each simplex of the triangulation, a bubble to replace \( u \) with a map homotopic to \( v \).

Several statements below are written with the assumptions that the maps under consideration are Lipschitz, but would be valid under less restrictive regularity assumptions.
However, writing and proving such more general statements would require to define Sobolev mappings on skeletons, an additional layer of technicality that can be avoided here having in mind as a goal the proof of \cref{theorem_Ehet_finite}, where we can start with smoothing the maps under consideration, as we will see at the end of the section. 

We start with the two following lemmas, that are concerned with suitable modifications near the skeleton \( \manifold{M}^{m - 1} \) of maps whose restrictions on \( \manifold{M}^{m - 1} \) are homotopic, and which are essentially a variant of the classical homotopy extension property.

\begin{lemma}
	\label{lemma_equalising_on_M_m_1}
	Let \(u \in \Lip \brk{\manifold{M}, \manifold{N}}\) and \(v \in \continuous \brk{\manifold{M}, \manifold{N}}\).
	If \( u\restr{\manifold{M}^{m - 1}} \) and \(v \restr{\manifold{M}^{m - 1}}\) are homotopic in \( \continuous \brk{\manifold{M}^{m - 1}, \manifold{N}}\), then there exists a mapping \(w \in \Lip \brk{\manifold{M}, \manifold{N}}\) such that \( w\restr{\manifold{M}^{m - 1}} = u\restr{\manifold{M}^{m - 1}} \) on \(\manifold{M}^{m - 1}\) and \(v\) and \(w\) are homotopic in \( \continuous \brk{\manifold{M}, \manifold{N}}\).
\end{lemma}

The key ingredient in the proof of~\cref{lemma_equalising_on_M_m_1} is the following cylinder insertion construction.

\begin{lemma}
	\label{lemma_homotopy}
	Let \(U \in \Lip \brk{\intvc{0}{1}\times \partial \Bset^m, \manifold{N}}\) and \(v \in \Lip \brk{\Bset^m, \manifold{N}}\).
	If 
	\[
	U\restr{\set{1} \times \partial \Bset^m}
	= v\restr{\partial \Bset^m}
	\]
	and if \(W \colon \intvc{0}{1} \times \Bset^m \to \manifold{N} \) is defined by
	\[
	W \brk{t,x}
	\defeq 
	\begin{cases}
		v \brk{\frac{2 x}{1+t}} & \text{if \(2\abs{x} \le 1 + t\),}\\
		U \brk[\Big]{\frac{1+t}{\abs{x}}-1, \frac{x}{\abs{x}}} & \text{if \(2\abs{x} \ge 1 + t\),}
	\end{cases}
	\]
	then
	\(W \in   \Lip \brk{\intvc{0}{1}\times \Bset^m, \manifold{N}}\),
	the map
	\[
	t \in \intvc{0}{1} \mapsto W \brk{t, \cdot}\in \mathrm{W}^{1, 1} \brk{\Bset^m, \manifold{N}}
	\]
	is continuous, for every \(t \in \intvc{0}{1}\),
	\[
	W \brk{t, \cdot}\restr{\partial \Bset^m} = U \brk{t, \cdot}\eqpunct{,}
	\]
	and 
	\[
	W \brk{1, \cdot} = v\eqpunct{.}
	\]
\end{lemma}
\begin{proof}%
	[Proof of \cref{lemma_homotopy}]
	The proof follows from the continuity in Sobolev spaces of suitable families of diffeomorphisms.
\end{proof}

\begin{proof}%
	[Proof of \cref{lemma_equalising_on_M_m_1}]
	By a classical smoothing by convolution argument, which preserves the homotopy classes of both \( v \) and \( v\restr{\manifold{M}^{m-1}}\), we can assume that \(v \in \Lip \brk{\manifold{M}, \manifold{N}} \).
	
	On the other hand, by a classical extension by convolution argument, there is some map \(U \in \Lip \brk{\intvc{0}{1}\times \manifold{M}, \manifold{N}} \) such that \( U\restr{\set{0}\times \manifold{M}} = u \).
	We will now focus on \( U \) on \( \set{0} \times \manifold{M}^{m-1} \).
	By the transitivity of homotopies, \(U \brk{1, \cdot}\restr{\manifold{M}^{m - 1}}\) is homotopic to \(v \restr{\manifold{M}^{m - 1}}\) in \(\continuous \brk{\manifold{M}^{m-1}, \manifold{N}}\); upon replacing this continuous homotopy with a Lipschitz one and then using this regularised homotopy to modify appropriately the values of \( U \) on \( \intvc{1/2}{1} \times \manifold{M}^{m - 1} \), we can assume further that \( U \brk{1, \cdot}\restr{\manifold{M}^{m - 1}} = v\restr{\manifold{M}^{m - 1}}\), while preserving the fact that \( U \in \Lip \brk{\intvc{0}{1}\times \manifold{M}^{m - 1}, \manifold{N}} \).
	
	We conclude by applying \cref{lemma_homotopy} to \(\sigma \times \intvc{0}{1}\) for every \(m\)-dimensional simplex \( \sigma \) of the triangulation \(\manifold{M}^m\) thanks to a suitable bi-Lipschitz homeomorphism between the simplex \(\sigma\) and the ball \(\Bset^m\).
\end{proof}

We now turn to homotopy results for maps that are sufficiently close in \( \lebesgue^m \) and have sufficiently small \( \sobolev^{1,m} \) energy, starting with the following oscillation estimate on a skeleton.

\begin{lemma}
\label{lemma_uniform_estimate_skeleton}
Given \(\ell \in \set{1, \dotsc, m - 1}\) satisfying \(\ell < p\), there exists a constant such that if \(u, v \in \Lip \brk{\manifold{M}^{\ell}, \manifold{N}}\),
then for almost every \(x \in \manifold{M}^\ell\),
\begin{equation*}
   d \brk{u \brk{x}, v \brk{x}}^p
   \le C
   \brk[\bigg]{\int_{\manifold{M}^{\ell}} \brk[\big]{\abs{\Deriv u}^p + \abs{\Deriv v}^p + d \brk{u, v}^p}}^{\frac{\ell}{p}}
   \brk[\bigg]{\int_{\manifold{M}^{\ell}} d \brk{u, v}^p}^{1 -\frac{\ell}{p}}\eqpunct{.}
\end{equation*}
\end{lemma}
Estimates like in \cref{lemma_uniform_estimate_skeleton} are classical in the study of Sobolev mappings \citelist{\cite{Luckhaus_1993}*{p. 453}\cite{White_1988}*{Th.\ 1.1}\cite{Hang_Lin_2003_II}*{(3.6)}}.
Its core ingredient is the Sobolev--Morrey embedding.

\begin{proof}[Proof of \cref{lemma_uniform_estimate_skeleton}]
Since \(\manifold{M}^\ell\) is a finite union of diffeomorphic images of simplices, it is sufficient to establish the estimate on an \(\ell\)-dimensional simplex, or equivalently on the unit ball \(\Bset^\ell\).
In this case, the Lipschitz regularity assumption implies that \( u \) and \( v \) are in particular Sobolev maps.

Assuming that \(u, v \in \sobolev^{1, p} \brk{\Bset^{\ell}, \manifold{N}}\), for almost every \(x \in \Bset^\ell\) and every \(r \in \intvl{0}{1}\), we have 
\begin{equation}
\label{eq_fahmaeGhahm4noevixuzooh2}
 d \brk{u \brk{x}, v \brk{x}}
 \le \smashoperator[r]{\fint_{\Bset^\ell \cap B_{r}\brk{x}}} d \brk{u \brk{x}, u \brk{y}} \dif y
 + \smashoperator[r]{\fint_{\Bset^\ell \cap B_{r}\brk{x}}} d \brk{u \brk{y}, v \brk{y}} \dif y
 +\smashoperator[r]{\fint_{\Bset^\ell \cap B_{r}\brk{x}}} d \brk{v \brk{y}, v \brk{x}} \dif y\eqpunct{.}
\end{equation}
We have then by the Sobolev--Morrey embedding 
\begin{equation}
\label{eq_gai5Ohh5ve2iash5ahRupho9}
 \smashoperator[r]{\fint_{\Bset^\ell \cap B_{r}\brk{x}}} d \brk{u \brk{x}, u \brk{y}} \dif y
 \le \Cl{cst_aey3lie0hep7Wae8auteush9} r^{1 - \frac{\ell}{p}}
 \brk[\bigg]{\int_{\Bset^\ell} \abs{\Deriv u}^p}^\frac{1}{p}
\end{equation}
and 
\begin{equation}
\label{eq_yiPhoh2iaRohCa8iel8Quuu4}
 \smashoperator[r]{\fint_{\Bset^\ell \cap B_{r}\brk{x}}} d \brk{v \brk{x}, v \brk{y}} \dif y
 \le \Cr{cst_aey3lie0hep7Wae8auteush9} r^{1 - \frac{\ell}{p}}
 \brk[\bigg]{\int_{\Bset^\ell} \abs{\Deriv v}^p}^\frac{1}{p},
\end{equation}
whereas by Jensen's inequality
\begin{equation}
\label{eq_eeY3quaex2nairohwai5yowu}
 \smashoperator[r]{\fint_{\Bset^\ell \cap B_{r}\brk{x}}} d \brk{u \brk{y}, v \brk{y}} \dif y
 \le \C r^{-\frac{\ell}{p}}\brk[\bigg]{\smashoperator[r]{\int_{\Bset^\ell}} d \brk{u, v}^{p}}^{\frac{1}{p}}.
\end{equation}
Defining \(t \in \intvo{0}{\infty}\) such that  
\[
 t^p \int_{\Bset^\ell} \abs{\Deriv u}^p + \abs{\Deriv v}^p
 = \int_{\Bset^\ell} d \brk{u, v}^{p}\eqpunct{,}
\]
the conclusion follows from \eqref{eq_fahmaeGhahm4noevixuzooh2}, \eqref{eq_gai5Ohh5ve2iash5ahRupho9}, \eqref{eq_yiPhoh2iaRohCa8iel8Quuu4}, and \eqref{eq_eeY3quaex2nairohwai5yowu}, with \(r \defeq  \min \brk{1, t}\).
\resetconstant
\end{proof}

With~\cref{lemma_uniform_estimate_skeleton} at hand, we are in position to state the following homotopy result for maps that are sufficiently close and have controlled \( \sobolev^{1,m} \) energy, and which is a variant of \cite{White_1988}*{Theorem~2.1}.

\begin{proposition}
\label{proposition_homotopy_skeleton_quantitative}
Assume that \( \manifold{N} \) is connected.
There exists \(\eta \in \intvo{0}{\infty}\) such that if \(u, v \in \Lip \brk{\manifold{M}, \manifold{N}}\) and if 
\[
   \brk[\bigg]{\int_{\manifold{M}} \abs{\Deriv u}^m + \abs{\Deriv v}^m}^{1-\frac{1}{m}}
   \brk[\bigg]{\int_{\manifold{M}} d \brk{u, v}^m}^{\frac{1}{m}} \le \eta\eqpunct{,}
\]
then
\begin{enumerate}[label=(\roman*)]
 \item 
 \label{it_kifoxoophah9Ushati9PhiaJ}
 \( u\restr{\manifold{M}^{m- 1}} \) and \( v\restr{\manifold{M}^{m- 1}} \) are homotopic in \(\continuous \brk{\manifold{M}^{m - 1}, \manifold{N}}\),
 \item 
 \label{it_oiRoov5aigi9aoMo2ziequ4I}
 there exists \( w \in \Lip \brk{\manifold{M}, \manifold{N}}\) that is homotopic in \(\continuous \brk{\manifold{M}, \manifold{N}}\) to \(v\), and satisfies \( u\restr{\manifold{M}^{m - 1}} = w\restr{\manifold{M}^{m - 1}}\).
\end{enumerate}
\end{proposition}

\begin{remark}
By Young's inequality, if 
\[
\int_{\manifold{M}} \abs{\Deriv u}^m + \abs{\Deriv v}^m + \frac{d \brk{u, v}^m}{\varepsilon^m}
\le \frac{\eta}{\varepsilon}\eqpunct{,}
\]
then the assumptions -- and hence the conclusion -- of \cref{proposition_homotopy_skeleton_quantitative} hold.
\end{remark}

\begin{proof}[Proof of \cref{proposition_homotopy_skeleton_quantitative}]
By the homotopy theory in \(\sobolev^{1, m}\brk{\manifold{M}, \manifold{N}}\), there exists \(\eta_0\) such that if 
\[ 
  \int_{\manifold{M}} \abs{\Deriv u}^m + \abs{\Deriv v}^m\le \eta_0\eqpunct{,}
\]
then \(u\) and \(v\) are both homotopic to a constant in \(\sobolev^{1, m} \brk{\manifold{M}, \manifold{N}}\) and the conclusion then follows from the connectedness of \( \manifold{N} \).

Otherwise, since the volume of \(\manifold{M}\) is finite and since \(\manifold{N}\) is compact, we have
\[
\int_{\manifold{M}} d \brk{u, v}^m 
\le \C \int_{\manifold{M}} \abs{\Deriv u}^m + \abs{\Deriv v}^m\eqpunct{.}
\]
Letting \(\Pi_{\manifold{M}}\) be the nearest point projection on \(\manifold{M} \subseteq \Rset^{\mu}\) and defining, for \(h \in \Rset^{\mu}\) with \(\abs{h}\le \delta\) and \(\delta\) sufficiently small,
\begin{align*}
 u_h &\defeq u \compose \Pi_{\manifold{M}} \brk{\cdot - h} \restr{\manifold{M}}&
 &\text{and}&
 v_h &\defeq v \compose \Pi_{\manifold{M}} \brk{\cdot - h} \restr{\manifold{M}}\eqpunct{,}
\end{align*}
we have, by
\begin{enumerate*}[label=(\(\textup{\roman*}\)), itemjoin={,\ }, itemjoin*={, and\ }]
	\item Hölder’s inequality
	\item Fubini's theorem combined with the fact that \( \Pi_{\manifold{M}} \brk{\cdot - h} \restr{\manifold{M}} \) is a diffeomorphism on \( \manifold{M} \) whenever \( h \) is sufficiently small and the change of variable
	\item the chain rule in Sobolev spaces,
\end{enumerate*}
\begin{equation}
\label{eq_waZ6Phuak4esheiGaeLeekoh}
\begin{split}
 &\int_{B_\delta}
 \brk[\bigg]{\int_{\manifold{M}^{m - 1}} \brk[\big]{\abs{\Deriv u_h}^m + \abs{\Deriv v_h}^m + d \brk{u_h, v_h}^m} }^{1 - \frac{1}{m}}
 \brk[\bigg]{\int_{\manifold{M}^{m - 1}} d \brk{u_h, v_h}^m }^{\frac{1}{m}} \dif h\\
 &\qquad\le 
 \brk[\bigg]{ \int_{B_\delta} \int_{\manifold{M}^{m - 1}} \brk[\big]{\abs{\Deriv u_h}^m + \abs{\Deriv v_h}^m + d \brk{u_h, v_h}^m} \dif h}^{1 - \frac{1}{m}}\\
 &\qquad \hspace{10em}\times
 \brk[\bigg]{ \int_{B_\delta} \int_{\manifold{M}^{m - 1}} d \brk{u_h, v_h}^m \dif h}^{\frac{1}{m}} \\
 &\qquad \le \C
 \brk[\bigg]{\int_{\manifold{M}} \brk[\big]{\abs{\Deriv u}^m + \abs{\Deriv v}^m + d \brk{u, v}^m} }^{1 - \frac{1}{m}}
 \brk[\bigg]{ \int_{\manifold{M}} d \brk{u, v}^m }^{\frac{1}{m}} \\
 &\qquad \le \C
 \brk[\bigg]{\int_{\manifold{M}} \brk[\big]{\abs{\Deriv u}^m + \abs{\Deriv v}^m}}^{1 - \frac{1}{m}}
 \brk[\bigg]{ \int_{\manifold{M}} d \brk{u, v}^m }^{\frac{1}{m}}\eqpunct{.}
\end{split}
\end{equation}
The exists thus \(h \in \Rset^\mu\) such that \(\abs{h}\le \delta\) and 
\begin{equation*}
\begin{split}
&\brk[\bigg]{\int_{\manifold{M}^{m - 1}} \brk[\big]{\abs{\Deriv u_h}^m + \abs{\Deriv v_h}^m + d \brk{u_h, v_h}^m} }^{1 - \frac{1}{m}}
 \brk[\bigg]{\int_{\manifold{M}^{m - 1}} d \brk{u_h, v_h}^m }^{\frac{1}{m}}\\
 &\qquad \le \C  \brk[\bigg]{\int_{\manifold{M}} \brk[\big]{\abs{\Deriv u}^m + \abs{\Deriv v}^m}}^{1 - \frac{1}{m}}
 \brk[\bigg]{ \int_{\manifold{M}} d \brk{u, v}^m }^{\frac{1}{m}} \dif h\eqpunct{.}
\end{split}
\end{equation*}
By \cref{lemma_uniform_estimate_skeleton}, we have 
\[
  \sup_{\manifold{M}^{m-1}} d \brk{u_h, v_h} \le \C \eta\eqpunct{,}
\]
and thus in particular, if \(\eta\) is chosen sufficiently small, then 
\( {u_h}\restr{\manifold{M}^{m - 1}}\) and \( {v_h}\restr{\manifold{M}^{m - 1}} \) are homotopic as continuous maps.
But now, the maps \( {u_h}\restr{\manifold{M}^{m - 1}}\) and \( {v_h}\restr{\manifold{M}^{m - 1}} \) are homotopic as continuous maps respectively to \( u\restr{\manifold{M}^{m - 1}} \) and \( v\restr{\manifold{M}^{m - 1}} \), by considering respectively the maps \( {u_{th}}\restr{\manifold{M}^{m - 1}}\) and \( {v_{th}}\restr{\manifold{M}^{m - 1}} \) for \( t \in \intvc{0}{1} \).
This concludes the proof of~\ref{it_kifoxoophah9Ushati9PhiaJ}.

The proof of \ref{it_oiRoov5aigi9aoMo2ziequ4I} essentially relies on the homotopy extension property, and is more precisely a direct consequence of \cref{lemma_equalising_on_M_m_1} above.
\resetconstant
\end{proof}

The last ingredient in the proof of~\cref{theorem_Ehet_finite} is the following insertion of bubble lemma.
 
\begin{lemma}
\label{lemma_bubble_from_2_disks}
Let \(u_0, u_1 \in \sobolev^{1, p} \brk{\Bset^m, \manifold{N}} \).
If
\[
\tr_{\partial \Bset^m} u_0 = \tr_{\partial \Bset^m} u_1\eqpunct{,}
\]
and if 
\[
 U 
 \brk{t, x}
 \defeq 
 \begin{cases}
   u_0 \brk{x} & \text{if \(\abs{x} \ge t\),}\\
  u_0 \brk{ t^2 x/\abs{x}^2} & \text{if \(t^2 \le \abs{x} \le t\),}\\
  u_1 \brk{x/t^2} & \text{if \(\abs{x} \le t^2\),}
 \end{cases}
\]
then
\begin{enumerate}[label=(\roman*)]
 \item\label{item:a1} the map 
 \(
  t \in \intvl{0}{1} \mapsto U \brk{t, \cdot} \in \sobolev^{1, p}\brk{\Bset^m, \manifold{N}}
 \) is continuous,
 \item\label{item:a2} for every \(t \in \intvl{0}{1}\), 
 \[
   \set{x \in \Bset^m \st U \brk{t, x}\ne u_0 \brk{x}}
   \subset B_t\eqpunct{,}
 \]
 \item\label{item:a3}
 for every \(t \in \intvl{0}{1}\),
 \[ \int_{\Bset^m} \abs{\Deriv U \brk{t, \cdot}}^p
 = \int_{\Bset^m \setminus B_t} \abs{\Deriv u_0\brk{x}}^p\brk[\bigg]{1 + \frac{t^{2 \brk{m - p}}}{\abs{x}^{2 \brk{m - p}}}} \dif x + t^{2 \brk{m - p}}\int_{\Bset^m} \abs{\Deriv u_1}^p\eqpunct{.}
 \]
 \end{enumerate}
\end{lemma}
\begin{proof}
Assertions~\ref{item:a1} and~\ref{item:a2} are straightforward.
For~\ref{item:a3}, one clearly has 
\[
 \int_{\Bset^m \setminus B_t} \abs{\Deriv U \brk{t, \cdot}}^p
 = \int_{\Bset^m \setminus B_t} \abs{\Deriv u_0}^m
\]
and 
\[
 \int_{B_{t^{2}}} \abs{\Deriv U \brk{t, \cdot}}^p
 = t^{2\brk{m-p}}\int_{\Bset^{m}} \abs{\Deriv u_1}^p\eqpunct{.}
\]
Moreover, since the transformation \(x \mapsto t^2 x/\abs{x}^{2}\) is conformal, we have
\[
\begin{split}
  \int_{B_t \setminus B_{t^2}} \abs{\Deriv U \brk{t, \cdot}}^p
  &= \int_{B_t \setminus B_{t^2}} \abs{\Deriv u_0 \brk{t^2 x/\abs{x}^2}}^p \frac{t^{2p}}{\abs{x}^{2 p}} \dif x \\
  &= \int_{\Bset^m \setminus B_{t}} \abs{\Deriv u_0 \brk{x}}^p
  \frac{t^{2\brk{p - m}}}{\abs{x}^{2 \brk{p - m}}}\dif x\eqpunct{.}
\end{split}
\]
This concludes the proof of item~\ref{item:a3}.
\end{proof}
\begin{proposition}
\label{proposition_Ehet_finite}
Let \(u, v  \in \Lip \brk{\manifold{M}, \manifold{N}}\).
If \( u\restr{\manifold{M}^{m- 1}}\) and \( v\restr{\manifold{M}^{m- 1}} \) are homotopic in \( \continuous \brk{\manifold{M}^{m - 1}, \manifold{N}}\), then 
\[
 \energy^{1, m}_{\mathrm{het}}
  \brk{u, v} < \infty\eqpunct{.}
\]
\end{proposition} 
\begin{proof}
In view of \cref{lemma_equalising_on_M_m_1}, we can assume that \( u\restr{\manifold{M}^{m- 1}} = v\restr{\manifold{M}^{m- 1}} \).
We then conclude by applying \cref{lemma_bubble_from_2_disks} to any \(m\)-dimensional simplex \(\sigma\) of the triangulation \(\manifold{M}\), up to a standard bi-Lipschitz equivalence between such a simplex and \(\Bset^m\).
More precisely, on each simplex \( \sigma \), we take \( u_{0} = u\restr{\sigma} \) and \( u_{1} = v\restr{\sigma} \) in \cref{lemma_bubble_from_2_disks}, and we let \( w_{t} \) be the map obtained by gluing the corresponding maps \( U \) on each \( \sigma \).
Taking \( p = m \), we deduce from~\ref{item:a3} that \( \int_{\manifold{M}} \abs{\Deriv w_{t}}^{p} \) remains bounded uniformly with respect to \( t \in \intvl{0}{1} \).
Meanwhile, it is clear from the construction of \( U \) in \cref{lemma_bubble_from_2_disks} that \( w_{t} \) is homotopic to \( v \) for every \( t \), while item~\ref{item:a2} ensures that \( w_{t} \to u \) a.e.\ as \( t \to 0 \).
The conclusion therefore follows by letting \( t \to 0 \).
\end{proof}

\begin{proof}[Proof of \cref{theorem_Ehet_finite}]
We first assume that~\ref{it_eKai9ahmoes3fe1biquah2pe} holds.
From \cref{proposition_heter_energy_controls_energy}, we deduce that \( u \in \sobolev^{1,m} \brk{\manifold{M},\manifold{N}} \).
We first observe that, by approximation in \( \sobolev^{1,m} \brk{\manifold{M}, \manifold{N}} \) combined with \cref{proposition_stability_homotopy_classes}, we can find a map \( \tilde{u} \in \continuous^{\infty} \brk{\manifold{M}, \manifold{N}} \) arbitrarily close to \( u \) with respect to the \( \sobolev^{1,m} \) norm and homotopic to \( u \) in \( \sobolev^{1,m} \brk{\manifold{M}, \manifold{N}} \) and hence in \( \VMO \brk{\manifold{M}, \manifold{N}} \).
Meanwhile, we deduce from \cref{proposition_heterotopic_restricts_to_smooth} that there exists \( \tilde{v} \in \continuous^{\infty} \brk{\manifold{M}, \manifold{N}} \), homotopic to \( v \), and arbitrarily close to \( u \) in \( \lebesgue^{m} \brk{\manifold{M}} \).
Invoking now \cref{proposition_homotopy_skeleton_quantitative}~\ref{it_oiRoov5aigi9aoMo2ziequ4I} -- as both continuous and \( \sobolev^{1,m} \) maps send connected sets to connected sets, by restricting to connected components of \( \manifold{M} \), we may assume without of generality \( \manifold{N} \) to be connected -- we obtain a map \( \tilde{w} \in \Lip\brk{\manifold{M}, \manifold{N}} \), homotopic to \( \tilde{v} \) and hence to \( v \), and satisfying \( \tilde{w}\restr{\manifold{M}^{m-1}} = \tilde{u}\restr{\manifold{M}^{m-1}} \).
We conclude by applying the homotopy extension property: the map \( \tilde{u} \), well-defined and continuous on the whole \( \manifold{M} \), and is continuously homotopic to \( v \) when restricted to \( \manifold{M}^{m-1} \), so extending this homotopy to \( \manifold{M} \) yields a map \( w \in \continuous \brk{\manifold{M}, \manifold{N}} \) that coincides with \( v \) on \( \manifold{M}^{m-1} \) and is homotopic to \( \tilde{u} \) in \( \continuous\brk{\manifold{M}, \manifold{N}} \), and hence to \( u \) in \( \VMO \brk{\manifold{M}, \manifold{N}} \).
This concludes the proof of~\ref{it_eiPh1baisahth4zieFee8xoh}.

If \ref{it_eiPh1baisahth4zieFee8xoh} holds,
by a standard approximation argument applied to the map \( w \) -- and using \cref{proposition_equivalence_homotopies} -- we can construct a mapping \(\Tilde{w} \in \continuous^{\infty} \brk{\manifold{M}, \manifold{N}} \subseteq \sobolev^{1, m} \brk{\manifold{M}, \manifold{N}}\)
such that \(u\) and \(\Tilde{w}\) are homotopic in \(\sobolev^{1, m}\brk{\manifold{M}, \manifold{N}}\) and such that \( \Tilde{w} \restr{\manifold{M}^{m-1}} \) is homotopic to \( v\restr{\manifold{M}^{m-1}} \) in \( \continuous \brk{\manifold{M}^{m-1}, \manifold{N}} \), and thus in \( \VMO \brk{\manifold{M}^{m-1}, \manifold{N}} \).
By continuity of the traces,
\(\smash{\tr_{\manifold{M}^{m-1}} u}\) and \(\smash{\tr_{\manifold{M}^{m  - 1}} \Tilde{w}} \) are homotopic in \( \smash{\sobolev^{1 - 1/m, m} \brk{\manifold{M}^{m - 1}, \manifold{N}}}\), and thus in \(\VMO \brk{\manifold{M}^{m - 1}} \), so that \ref{it_aitoo2kePh0Ue2Teitaechae} holds.

Conversely, assume that \ref{it_aitoo2kePh0Ue2Teitaechae} holds.
By an approximation argument, we may assume both \( u \) and \( v \) to be smooth.
For \( v \), as it is continuous, this is a straightforward argument.
Concerning \( u \), we notice that, by continuity of the trace, if a sequence \( \brk{u_{j}}_{j \in \Nset}\) of maps in \(\continuous^{\infty} \brk{\manifold{M}, \manifold{N}} \) converges to \( u \) in \( \sobolev^{1,m} \brk{\manifold{M}, \manifold{N}} \), then the sequence \( \brk{{u_{j}}\restr{\manifold{M}^{m-1}}}_{j \in \Nset} \) converges to \( u\restr{\manifold{M}^{m-1}} \) in \( \sobolev^{1-1/m,m}\brk{\manifold{M}^{m-1}, \manifold{N}} \), and hence in \( \VMO \brk{\manifold{M}^{m-1}, \manifold{N}} \).
But this implies that, for \( j \) sufficiently large, \( {u_{j}}\restr{\manifold{M}^{m-1}} \) and \( u\restr{\manifold{M}^{m-1}} \) are homotopic in \( \VMO \brk{\manifold{M}^{m-1}, \manifold{N}} \) -- stricly speaking, this is not included in Brezis and Nirenberg~\cite{Brezis_Nirenberg_1995} which work on smooth manifolds; we refer the reader to \cite{Detaille_Mironescu_Xiao_2025} for a statement and proof on doubling metric measure spaces, which includes our setting here. 
But then, \ref{it_eKai9ahmoes3fe1biquah2pe} is a consequence of \cref{proposition_Ehet_finite}.
\end{proof}

\section{Improved upper estimate}
\subsection{Topological disparities}
\label{section_topological_disparity}
This section is devoted to an estimate on the heterotopic energy, which in particular implies the upper estimate in~\cref{theorem_intro_heterotopic_disparity}.
We start by introducing the concept of \emph{topological disparity}, and collecting some of its basic properties.

Given two mappings \(u, v \in \continuous \brk{\overline{\Bset^m}, \manifold{N}}\) such that \(u = v\) on \(\partial \Bset^m\), we define the topological disparity
\[
 \sqb{u, v, \Bset^m}
 \in \pi_{m} \brk{\manifold{N}, u \brk{0}}
\]
as the homotopy class of the mapping
\(w \colon \overline{\Bset^m} \to \manifold{N}\) defined as
\[
 w \brk{x}
 \defeq
 \begin{cases}
  u \brk{4 \brk{1 - \abs{x}} x} &
  \text{if \(\abs{x} \ge 1/2\),}\\
  v \brk{2 x} &\text{if \(\abs{x} \le 1/2\),}
 \end{cases}
\]
relatively to \(\partial \Bset^m\). (The definition ensures that \(w\restr{\partial \Bset^m} = u \brk{0}\).)

If \(\rho\) is sufficiently small, if we fix a diffeomorphism mapping \(a\) to \(0\) and \(\Bar{B}_{\rho} \brk{a}\) to \(\overline{\Bset^m}\), we can define for \(u, v \in \continuous \brk{\manifold{M}, \manifold{N}}\) such that \(u \restr{\partial B_\rho\brk{a}} = v \restr{\partial B_\rho\brk{a}}\),
\[
 \sqb{u, v, B_{\rho} \brk{a}} \in \pi_{m} \brk{\manifold{N}, u \brk{a}}
\]
accordingly, and it is a well-defined homotopy class (depending on the orientation of the diffeomorphism).

We first prove that having zero topological disparity is a sufficient condition for being homotopic.

\begin{lemma}
\label{lemma_disparity_trivial_homotopic}
Let \(u, v\in \continuous \brk{\manifold{M}, \manifold{N}}\).
If \(u = v\) on \(\manifold{M} \setminus
B_{\rho} \brk{a}\) and if 
\begin{align*}
 \sqb{u, v, B_\rho \brk{a}}
&= 0&
&\text{in }\pi_m \brk{\manifold{N}, u \brk{a}}\eqpunct{,}
\end{align*}
then \(u\) and \(v\) are homotopic relatively to \(\manifold{M}\setminus B_\rho\brk{a}\).
\end{lemma}
\begin{proof}
Since \(\sqb{u, v, B_\rho \brk{a}} = 0\), it follows from the definition of the topological disparity that \(u\) and \(v\) are homotopic relatively to \(\partial B_\rho \brk{a}\);
since they coincide outside of \( B_{\rho} \brk{a} \), they are homotopic relatively to \(\manifold{M}\setminus B_\rho\brk{a}\).
To obtain the homotopy relatively to \( \partial B_{\rho} \brk{a} \), one may argue as follows.
First slightly deform \( u \) inside \( B_{\rho} \brk{a} \) to make it constantly equal to \( u\brk{0} \) near \( a \).
By definition of the topological disparity as the homotopy class of the map \( w \) defined as above on \( \overline{\Bset^m} \) relatively to \( \partial\Bset^{m} \), we can now perform a further homotopy in this constancy region to replace it by a scaled copy of \( w \).
But now, the map obtained through this process is homotopic to \( v \) relatively to \( \partial B_{\rho} \brk{a} \), by cancelling the inverted copy of \( u \) coming from \( w \) with \( u \) and dilating accordingly the copy of \( v \) coming from \( w \) to fill in \( B_{\rho} \brk{a} \).
\end{proof}

Without imposing \( u \) and \( v \) being homotopic \emph{relatively to \(\manifold{M}\setminus B_\rho\brk{a}\)}, the converse of \cref{lemma_disparity_trivial_homotopic} fails in general. 
For example, let us take the projective space \(\manifold{M} = \RPset^{2n}\) and the sphere \(\manifold{N} = \Sset^{2n}\), and assume that \(u = v\) on \(\manifold{M} \setminus B_{\rho} \brk{a}\) and that \(\sqb{u, v, B_{\rho}\brk{a}}\) is a map of Brouwer degree \(2\).
Then, obviously \(\sqb{u, v, B_{\rho}\brk{a}} \neq 0 \). 
On the other hand, the maps obtained from \( u \) respectively by making it constantly equal to \( u \brk{0} \) around \( a \) and by inserting a copy of \( w \) near \( a \) along the lines of the proof of \cref{lemma_disparity_trivial_homotopic} -- the first map being homotopic to \( u \) and the second to \( v \), as we saw in the aforementioned proof -- \emph{are homotopic} to each other.
This can be seen by first splitting \( w \) into two maps of Brouwer degree \(1\) -- this can be done inside \( B_{\rho} \brk{a} \) -- and then moving one along an orientation reversing loop in \( \RPset^{2n} \) to cancel it with the other one -- this requires to leave \( B_{\rho} \brk{a} \). 
This shows that \( u \) and \( v \) are homotopic in \( \manifold{M} \), while \(\sqb{u, v, B_{\rho}\brk{a}} \neq 0 \).

A straightforward property of the disparity energy is that it does not depend on the choice of a ball outside of which \( u \) and \( v \) coincide -- provided that one makes consistent orientation choices.

\begin{lemma}
\label{lemma_disparity_change_radius}
Let \(u, v\in \continuous \brk{\manifold{M}, \manifold{N}}\).
If \( \rho < \sigma \) and if \(u = v\) on \(\Bar{B}_{\sigma}\brk{a} \setminus
B_{\rho} \brk{a}\), then 
\begin{align*}
 \sqb{u, v, B_\rho \brk{a}}
&= \sqb{u, v, B_\sigma \brk{a}}&
&\text{in }\pi_m \brk{\manifold{N}, u \brk{a}}\eqpunct{,}
\end{align*}
provided the orientations on \(\Bar{B}_{\rho} \brk{a}\) and \(\Bar{B}_\sigma \brk{a}\) are  compatible with the inclusion.
\end{lemma}
\begin{proof}
In view of the inclusion \(\Bar{B}_{\rho} \brk{a} \subseteq \Bar{B}_{\sigma}\brk{a}\), there is a homotopy between the diffeomorphism between \(\overline{\Bset^m}\) and \(\Bar{B}_\rho\brk{a}\) and the one between \(\overline{\Bset^m}\) and \(\Bar{B}_\sigma\brk{a}\) that can be used to construct a suitable homotopy.
\end{proof}

\cref{lemma_disparity_change_radius} essentially asserts that we may enlarge the ball on which we consider the topological disparity.
As a converse, we show that can always work with arbitrarily smaller balls up to a suitable homotopy -- which is not permitted by \cref{lemma_disparity_change_radius} as it requires that the maps \( u \) and \( v \) coincide outside of the ball under consideration.

\begin{lemma}
\label{lemma_disparity_modify_radius}
Let \(u, v\in \continuous \brk{\Bar{B}_{\rho} \brk{a}, \manifold{N}}\).
If \(u = v\) on \(\partial B_{\rho} \brk{a}\) and if \( \sigma < \rho\), then there exists \(w \in \continuous \brk{\Bar{B}_{\rho}\brk{a}, \manifold{N}}\) homotopic to \( v \) relatively to \( \partial B_{\rho} \brk{a} \) such that \( w = u \) in \( \Bar{B}_{\rho} \brk{a} \setminus B_\sigma \brk{a}\) and
\begin{align*}
\sqb{u, w, B_\sigma \brk{a}}
 &= \sqb{u, v, B_\rho \brk{a}} &
&\text{in }\pi_m \brk{\manifold{N}, u \brk{a}}\eqpunct{,}
\end{align*}
provided the orientations on \(\Bar{B}_{\rho} \brk{a}\) and \(\Bar{B}_\sigma \brk{a}\) are  compatible with the inclusion.
\end{lemma}
\begin{proof}
Choose \( 0 < \sigma' < \sigma \), and define 
\[
	w\brk{x} = 
	\begin{cases}
		u \brk{x} & \text{if \( x \in \Bar{B}_{\rho} \brk{a} \setminus B_\sigma \brk{a} \),} \\
		u \brk{\lambda\brk{\abs{x}}x} & \text{if \( x \in \Bar{B}_{\sigma} \brk{a} \setminus B_{\sigma'} \brk{a} \),}\\
		v \brk{\rho x/\sigma'} & \text{if \( x \in \Bar{B}_{\sigma'} \brk{a} \),}
	\end{cases}
\]
where \( \lambda \colon \intvc{\sigma'}{\sigma} \to \Rset \) is a nonincreasing function such that \( \lambda\brk{\sigma'} = \rho/\sigma' \) and \( \lambda\brk{\sigma} = 1 \); the map \( w \) can readily be checked to satisfy the required properties.
\end{proof}

We now turn to a series of basic but useful properties related to moving, adding, or merging topological disparities.
For this purpose, we recall the definition of the action of a path on a homotopy class.

\begin{definition}
\label{definition_action_path_homotopy}
	If \( f \colon \overline{\Bset^m}\to \manifold{N} \) is a continuous map constantly equal to \( b \) on \( \partial\Bset^{m} \) and \( \zeta \colon \intvc{0}{1} \to \manifold{N} \) is a continuous path such that \( \zeta\brk{1} = b \), then the class \( \zeta_*\sqb{f} \in \pi_{m}\brk{\manifold{N}, \zeta \brk{0}} \) is the homotopy class of the map
	\[
	\overline{\Bset^m} \ni x \mapsto 
	\begin{cases}
		f\brk{2x} & \text{if \( \abs{x} \leq 1/2 \),}\\
		\zeta\brk{2-2\abs{x}} & \text{if \( \abs{x} \geq 1/2 \).}
	\end{cases}
	\]
\end{definition}
This corresponds to adding the path \( \zeta \) radially around the map \( f \).

\begin{lemma}[Adding topological disparities]
\label{lemma_superpose_disparities}
If \(u, v, w \in \continuous \brk{\Bar{B}_\rho\brk{a}, \manifold{N}}\), if  \(u = v = w \) on \(\partial
B_{\rho} \brk{a}\), if \(\gamma \in \continuous \brk{\intvc{0}{1}, \Bar{B}_\rho\brk{a}}\)
with \(\gamma \brk{1} = a\) and \(\gamma \brk{0} \in \partial B_\rho \brk{a}\),
then
\begin{align*}
 \brk{u \compose \gamma}_* \sqb{u, w, B_\rho \brk{a}}
 &= \brk{u \compose \gamma}_* \sqb{u, v, B_\rho \brk{a}}
 + \brk{v \compose \gamma}_* \sqb{v, w, B_\rho \brk{a}}
 &\text{in }\pi_m \brk{\manifold{N}, u \brk{\gamma\brk{0}}} \eqpunct{.}
\end{align*}
\end{lemma}

In the statement of the above lemma, the role of the assumption that \( \gamma \brk{0} \in \partial B_\rho \brk{a} \) is to ensure that the boundary values of the different homotopy classes coincide.

If \(u, v, w \in \continuous \brk{\manifold{M}, \manifold{N}}\), the conclusion still holds when \(u \compose \gamma = v \compose \gamma\) on \(\gamma^{-1} \brk{\manifold{M} \setminus B_\rho\brk{a}}\).

\begin{proof}[Proof of \cref{lemma_superpose_disparities}]
This follows from the definition of the sum of homotopy classes and a homotopy argument.
\end{proof}

\begin{lemma}[Reciprocal topological disparities]
	\label{lemma_disparity_reciprocal}
	Given \(u, v \in \continuous \brk{\Bar{B}_\rho\brk{a}, \manifold{N}}\),
	if \(u = v\) on \(\partial
	B_{\rho} \brk{a}\), and \(\gamma \in \continuous \brk{\intvc{0}{1}, \Bar{B}_\rho\brk{a}}\)
	with \(\gamma \brk{1} = a\) and \(\gamma \brk{0} \in \partial B_\rho \brk{a}\), then
	\begin{align*}
		\brk{u \compose \gamma}_* \sqb{u, v, B_\rho \brk{a}}
		&= - \brk{v \compose \gamma}_* \sqb{v, u, B_\rho \brk{a}}
		&\text{in }\pi_m \brk{\manifold{N}, u \brk{\gamma \brk{0}}} \eqpunct{.}
	\end{align*}
\end{lemma}
\begin{proof}
	We apply \cref{lemma_superpose_disparities} above with \( w = u \).
\end{proof}

\begin{lemma}[Moving topological disparities]
\label{lemma_moving_disparities}
Given \(\gamma \in \continuous \brk{\intvc{0}{1}, \manifold{M}}\) and \(u, v \in \continuous \brk{\manifold{M}, \manifold{N}}\),
if \(u = v \) on \( \manifold{M} \setminus B_{\rho} \brk{\gamma\brk{0}}\),
then there exists \( w  \in \continuous \brk{\manifold{M}, \manifold{N}} \) such that \(w = u \) on \(\manifold{M} \setminus B_{\rho} \brk{\gamma \brk{1}}\), \( w \) is homotopic to \( v \),
and
\begin{align*}
 \sqb{u, v, B_\rho\brk{\gamma \brk{0}}}
 &= \brk{u \compose \gamma}_*
 \sqb{u, w, B_\rho\brk{\gamma \brk{1}}}
 &
&\text{in }\pi_m \brk{\manifold{N}, u \brk{\gamma \brk{0}}}\eqpunct{,}
\end{align*}
where the orientation of \(\Bar{B}_{\rho} \brk{\gamma\brk{0}}\) is transported from  \(\Bar{B}_{\rho}\brk{\gamma\brk{1}}\) through \(\gamma\).
\end{lemma}

In the above lemma, if \( \gamma \) is smooth, then the transport of the orientation is defined via the parallel transport of a basis along \( \gamma \).
The definition is then extended to continuous curves by a standard approximation procedure.

In the particular case where \(\manifold{M}\) is not orientable, \(\gamma\brk{0} = \gamma\brk{1} = a\), and \(\gamma\) reverses the orientation, one should understand that opposite orientations are taken on both sides, and thus 
\begin{align*}
  \sqb{u, v, B_\rho\brk{a}}
  &= - \sqb{u, w, B_\rho\brk{a}}
  &\text{in }\pi_m \brk{\manifold{N}, u \brk{a}}
  \eqpunct{.}
\end{align*}

This sort of transport can be observed (without the orientation part) in \cite{Olum_1950}. In \cite{Baues_1977}*{Ch.\ 4}, a similar work is done with the universal coverings of \(\manifold{M}\) and \(\manifold{N}\): in other words, one writes \(\gamma = \pi_{\manifold{M}} \compose \lifting{\gamma}\), and then the orientation is managed through the orientation induced from \(\lifting{\gamma} \brk{0}\), since the universal covering space \(\lifting{\manifold{M}}\) is simply connected and thus orientable.

\begin{proof}[Proof of \cref{lemma_moving_disparities}]
By parallel transport and properties of the exponential map, there is a mapping 
\(K \in \continuous \brk{\intvc{0}{1}\times \overline{\Bset^m}, \manifold{M}}\) such that 
\(K \brk{\cdot, 0} = \gamma\) and for every \(t \in \intvc{0}{1}\), \(K \brk{t, \cdot}\) is a homeomorphism to \(\Bar{B}_{\rho}\brk{\gamma\brk{t}}\).
By the homotopy extension property, there exists a map \(W \in \continuous\brk{\intvc{0}{1}\times \overline{\Bset^m}, \manifold{N}}\)
such that 
\(W \brk{0, \cdot} = v \compose K \brk{0, \cdot}\) and for every \(t \in \intvc{0}{1}\)
\(
 W \brk{t, \cdot}\restr{\partial \Bset^m}
 = u \compose K \brk{t, \cdot}\restr{\partial \Bset^m}
\).
We define now 
\[
 V \brk{t, x} = 
 \begin{cases}
  W \brk{t, K \brk{t, \cdot}^{-1}\brk{x}} &\text{if \(x \in B_{\rho}\brk{\gamma \brk{t}}\),}\\
  u \brk{x} & \text{otherwise.}
 \end{cases}
\]
By construction, we have \(V \brk{0, \cdot} = v\) and \(w \defeq V \brk{1, \cdot}\) satisfies the conclusion.
\end{proof}

\begin{lemma}[Homotopy variations of disparities]
\label{lemma_discrepancy_homotopy_transport}
Let \(u_0, v_0 \in\continuous \brk{\manifold{M}, \manifold{N}}\).
If \(u_0 = v_0\) in \(\manifold{M}\setminus \bigcup_{i = 1}^I B_{\rho}\brk{a_i}\)
 and if \(u_1 \in \continuous \brk{\manifold{M}, \manifold{N}}\) is homotopic to \(u_0\), then there exist \(v_1 \in \continuous \brk{\manifold{M}, \manifold{N}}\) homotopic to \(v_0\) such that \(u_1 = v_1\) in \(\manifold{M} \setminus \bigcup_{i = 1}^I B_\rho \brk{a_i}\), and mappings \(\zeta_i \in \continuous \brk{\intvc{0}{1}, \manifold{N}}\) such that \(\zeta_i \brk{0} = u_0 \brk{a_i}\), \(\zeta_i \brk{1} = u_1 \brk{a_i}\), and
\begin{align*}
 \sqb{u_0, v_0, B_\rho \brk{a_i}}
 &= \brk{\zeta_i}_* \sqb{u_1, v_1, B_\rho \brk{a_i}}&\text{in }\pi_m \brk{\manifold{N}, u_0 \brk{a_i}} \eqpunct{.}
\end{align*}
\end{lemma}
\begin{proof}
By assumption, there is some \(U \in \continuous \brk{\intvc{0}{1} \times \manifold{M}, \manifold{N}}\) such that \(U\brk{0, \cdot} = u_0\) and \(U \brk{1, \cdot} = u_1\). Since \(U \brk{0, \cdot} = v_0\) in \(\manifold{M}\setminus \bigcup_{i = 1}^I B_{\rho}\brk{a_i}\), applying the homotopy extension property to the map \( v_{0} \) and the homotopy \[ U\restr{\intvc{0}{1} \times \manifold{M}\setminus \bigcup_{i = 1}^I B_{\rho}\brk{a_i}} \eqpunct,\] we find a map \(V\in \continuous \brk{\manifold{M}, \manifold{N}}\) such that \(V = U\) on \(\intvc{0}{1}\times \manifold{M}\setminus \bigcup_{i = 1}^IB_\rho \brk{a_i}\) and \(V \brk{0, \cdot} = v_0\).
We let then \(v_1\defeq V \brk{1, \cdot}\) and \(\zeta_i \defeq U \brk{\cdot, a_i} \).
\end{proof}

\begin{lemma}[Merging topological disparities]
\label{label_merging}
Given \(\gamma_i \in \continuous \brk{\intvc{0}{1}, \manifold{M}}\), we set \(\gamma_i \brk{0} = a_0 \) and \(\gamma_i \brk{1} = a_i\).
Given \( u, v \in \continuous \brk{\manifold{M}, \manifold{N}} \),
if \(u = v\) on \(\manifold{M}\setminus \bigcup_{i = 1}^{I} B_{\rho} \brk{a_{i}}\),
and if the balls \(\brk{B_\rho\brk{a_i}}_{1 \le i \le I}\) are disjoint,
then there exists \(w \in \continuous \brk{\manifold{M}, \manifold{N}}\) such that \( w = u \) on \(\manifold{M} \setminus B_{\rho}\brk{a_{0}}\),
\(w\) is homotopic to \(v\), and
\begin{align*}
 \sqb{u, w, B_\rho\brk{a_0}}
& = \sum_{i = 1}^I \brk{u \compose \gamma_i}_*
 \sqb{u, v, B_{\rho}\brk{a_i}}&
 &\text{in }\pi_m \brk{\manifold{N}, u \brk{a_0}} \eqpunct{,}
\end{align*}
where the orientation of \(\Bar{B}_{\rho} \brk{a_i}\) corresponds to the one transported from  \(\Bar{B}_{\rho}\brk{a_0}\) through \(\gamma_i\).
\end{lemma}
\begin{proof}
Thanks to \cref{lemma_moving_disparities}, we can assume that \(B_{\rho}\brk{a_0}\) is also disjoint from the balls \(B_\rho\brk{a_1}, \dotsc, B_{\rho}\brk{a_I}\). Since \(m \ge 2\), we can assume, up to a suitable homotopy, that for every \(i, j \in \set{1, \dotsc, I}\) with \(i < j\) and for every \(t \in \intvc{0}{1}\),
\begin{equation}
\label{eq_phooch7Oi9aich9chiaYae3o}
\gamma_i \brk{t} \not \in B_{\rho}\brk{a_j}\eqpunct{.}
\end{equation}
We first define \(w_0 = v\).
For every \(i \in \set{1, \dotsc, I}\), given \(w_{i - 1} \in \continuous \brk{\manifold{M}, \manifold{N}}\) homotopic to \(v\) such that
\(
 w_{i - 1} = u
\) on \(\manifold{M}\setminus \brk{B_\rho\brk{a_0}\cup \bigcup_{j = i}^I B_{\rho}\brk{a_j}} \)
and \(w_{i - 1} = v\) on \(\bigcup_{j = i}^I B_{\rho}\brk{a_j}\), we define
\begin{equation}
\label{eq_heim0Deegh0fai7OhxeeveeS}
 u_i \defeq
 \begin{cases}
  w_{i - 1} &\text{in \(\manifold{M}\setminus B_{\rho} \brk{a_i}\),}\\
  u &\text{in \(B_{\rho} \brk{a_i}\),}
 \end{cases}
\end{equation}
so that
\(u_i = u\)  on \(\manifold{M}\setminus \brk{B_\rho\brk{a_0}\cup \bigcup_{j = i + 1}^I B_{\rho}\brk{a_j}} \) and \(u_i = v\) on \(\bigcup_{j = i + 1}^I B_{\rho}\brk{a_i}\).
We get from \cref{lemma_moving_disparities} applied to the reverse path to \(\gamma_i\) a mapping \(w_{i} \in \continuous \brk{\manifold{M}, \manifold{N}}\) such that
(i) \(w_i = u_i\) on \(\manifold{M}\setminus B_{\rho}\brk{a_0}\) and thus \(
 w_{i} = u
\) on \(\manifold{M}\setminus \brk{B_\rho\brk{a}\cup \bigcup_{j = i + 1}^I B_{\rho}\brk{a_j}} \)
and \(w_{i} = v\) on \(\bigcup_{j = i + 1}^I B_{\rho}\brk{a_j}\),
(ii) \(w_i\) is homotopic to \(w_{i - 1}\) and thus to \(v\),
and (iii)
\begin{align}
\label{eq_ne8oiTh1Ri7ePhohdamahShe}
   \sqb{u_i, w_{i}, B_\rho\brk{a_0}}
   &= \brk{u_i \compose \gamma_i}_*  \sqb{u_i, w_{i - 1}, B_\rho\brk{a_i}}&
   &\text{in } \pi_m \brk{\manifold{N}, u_i \brk{a_i}}
   \eqpunct{.}
\end{align}
Since \(u \compose \gamma_i = u_i \compose \gamma_i\) in \(\gamma_i^{-1}\brk{\manifold{M} \setminus B_\rho\brk{a}}\) in view of \eqref{eq_phooch7Oi9aich9chiaYae3o}, applying \cref{lemma_superpose_disparities}, we have
by \eqref{eq_heim0Deegh0fai7OhxeeveeS} and \eqref{eq_ne8oiTh1Ri7ePhohdamahShe}
\[
\begin{split}
 \brk{u \compose \gamma_i}_*^{-1}  \sqb{u, w_{i}, B_\rho\brk{a_0}}
 &=
 \brk{u \compose \gamma_i}_*^{-1} \sqb{u, u_{i}, B_\rho\brk{a_0}}
 +
 \brk{u_i \compose \gamma_i}_*^{-1}  \sqb{u_i, w_{i}, B_\rho\brk{a_0}}\\
 &= \brk{u \compose \gamma_i}_*^{-1}  \sqb{u, u_{i}, B_\rho\brk{a_0}}
 + \sqb{u_i, w_{i - 1}, B_\rho\brk{a_i}}\\
 & = \brk{u \compose \gamma_i}_*^{-1}  \sqb{u, w_{i - 1}, B_\rho\brk{a_0}}
 + \sqb{u, v, B_\rho\brk{a_i}}\eqpunct{,}
\end{split}
\]
and thus
\[
 \sqb{u, w_{i}, B_\rho\brk{a_{0}}}
 =  \sqb{u, w_{i - 1}, B_\rho\brk{a_{0}}}
 + \brk{u \compose \gamma_i}_* \sqb{u, v, B_\rho\brk{a_i}}\eqpunct{.}
\]
By induction, this implies that
\[
	\sqb{u, w_{I}, B_\rho\brk{a_{0}}}
	=
	\sqb{u, w_{0}, B_\rho\brk{a_{0}}} + \sum_{i = 1}^I \brk{u \compose \gamma_i}_*
	\sqb{u, v, B_{\rho}\brk{a_i}}\eqpunct{.}
\]
We conclude then by letting \( w \defeq w_{I} \), and using the fact that \( u = v \) on \( B_{\rho} \brk{a_{0}} \) as we ensured that \(B_{\rho}\brk{a_0}\) is also disjoint from the balls \(B_\rho\brk{a_1}, \dotsc, B_{\rho}\brk{a_I}\), so that \( \sqb{u, w_{0}, B_\rho\brk{a_{0}}} = \sqb{u, v, B_\rho\brk{a_{0}}} = 0 \).
\end{proof}

\subsection{The topological energy}
\label{section_topological_energy}

Our next aim is to define the energy associated with topological disparities.
For this purpose, we first define the energy of a homotopy class.
We define, for \(\sigma  \in \pi_{m} \brk{\manifold{N}, b}\), the \emph{topological energy}
\begin{equation}
\label{eq_of7vae7eiquooDooh8nohnai}
   \energy^{1, m}_{\mathrm{top}} \brk{\sigma}
   \defeq 
   \inf
   \set[\bigg]{\int_{\Bset^m} \abs{\Deriv f}^m \st f \in \sobolev^{1, m}\brk{\Bset^m, \manifold{N}} \text{ and } f  \in \sigma}\eqpunct{.}
\end{equation}

When \(\manifold{N} = \Sset^m\), the topological energy is related to  \(\deg \colon \pi_{m}\brk{\Sset^m, b} \to \Zset\), defined as the Brouwer degree, through an exact formula:

\begin{proposition}
\label{proposition_Etop_sphere}
	For every \( \sigma \in \pi_{m} \brk{\Sset^{m},b} \),
	\[
		\energy^{1, m}_{\mathrm{top}} \brk{\sigma}
		=
		m^{m/2}\abs{\Sset^{m}} \abs{\deg{\sigma}} \eqpunct{.}
	\]
\end{proposition} 

This formula is well-known, but we provide an argument for the sake of completeness.
We also refer to~\cite{Lemaire_1978}*{\S 8} for the case where \( m = 2 \), and~\cite{Mucci_2012}*{Proposition~7.1} for a closely related result and argument in any dimension.

\begin{proof}[Proof of \cref{proposition_Etop_sphere}]
	Let \( f \in \sobolev^{1, m}\brk{\Bset^m, \Sset^{m}} \) be any map such that \( f \in \sigma \).
	From the arithmetico-geometric inequality, we deduce the pointwise estimate
	\begin{align*}
		\jac{f} = \abs{\det \Deriv f}
		&\leq
		\frac{\abs{\Deriv f}^{m}}{m^{m/2}}&
		&\text{in \(\Bset^m\)}\eqpunct{.}
	\end{align*}
	On the other hand, from the Kronecker integral formula for the degree 
	(see for example \cite{Bourgain_Brezis_Mironescu_2005}*{Remark 0.7} or~\cite{Dinca_Mawhin_2021}), it holds that
	\[
		\int_{\Bset^{m}} \jac{f}
		\geq
		\abs{\Sset^{m}} \abs{\deg{\sigma}}\eqpunct{.}
	\]
	Taking the infimum over all such \( f \in \sigma \), we deduce the lower bound
	\begin{equation}
	\label{eq_aa7a8d4e964a043b}
		\energy^{1, m}_{\mathrm{top}} \brk{\sigma}
		\geq
		m^{m/2}\, \abs{\Sset^{m}} \, \abs{\deg{\sigma}}\eqpunct{.}
	\end{equation}
	
	To prove the upper bound, we start with the case where \( \deg{\sigma} = 1 \).
	In this case, given \( \varepsilon > 0 \) and letting \( f \colon \Bset^{m} \to \Sset^{m} \) be a an almost conformal equivalence between \( \Bset^{m} \) and \( \Sset^{m} \) -- this can be achieved by the means of a truncation of a rescaled stereographical projection from \(\Rset^m\) to \(\Sset^m\) -- we obtain
	\[
		\int_{\Bset^{m}} \abs{\Deriv f}^{m}
		=
		m^{m/2} \abs{\Sset^{m}} + \varepsilon \eqpunct{,}
	\]
	showing that~\eqref{eq_aa7a8d4e964a043b} is actually an equality.
	
	The case of an arbitrary degree then follows via~\eqref{eq_thuav3woo1thie4eeNgeiNei} below, the idea being to construct almost minimizing competitors by patching together scaled copies of the above map.
\end{proof}

As similar formula holds for the projective plane; see~\cite{Mucci_2012_Projective}*{Proposition~3.8}

\begin{proposition}
\label{proposition_Etop_projective}
	For every \( \sigma \in \pi_{m} \brk{\RPset^{m},b} \),
	\[
		\energy^{1, m}_{\mathrm{top}} \brk{\sigma}
		=
		m^{m/2}\,2\abs{\RPset^{m}}\, \abs{\deg{\sigma}}\eqpunct{,}
	\]
	where \(\deg\) denotes the degree as a map to \(\RPset^m\), defined as the degree of the lifting when \(m\) is odd and its absolute value otherwise \cite{Mucci_2012_Projective}.
\end{proposition}
The proof of~\cref{proposition_Etop_projective} follows the same lines as the proof of~\cref{proposition_Etop_sphere}.
The extra factor \( 2 \) comes from the use of a lifting \( \Tilde{f} \colon \Bset^{m} \to \Sset^{m} \) of a map \( f \colon \Bset^{m} \to \RPset^{m} \), as each point in \( \RPset^{m} \) is covered twice by the covering map \( \Sset^{m} \to \RPset^{m} \).

As we already mentioned in the introduction, there is a more general phenomenon connecting the topological energy to the minimal area needed to realize the corresponding homotopy class.
We first recall the definition of the area of a map, and define accordingly the area of a homotopy class.

\begin{definition}
\label{definition_area_map}
	Let \( f \colon \overline{\Bset^{m}} \to \manifold{N} \) be a Lipschitz map.
	The area of \( f \) is defined as
	\[
		\area{f} = \int_{\manifold{N}} N\brk{f\vert y} \dif y\text{,}
	\]
	where \( N\brk{f \vert y} \) is the multiplicity of \( f \) at \( y \), defined as the number of elements of \( f^{-1}\brk{\set{y}} \), possibly infinite.
	The area of a homotopy class \( \sigma \in \pi_{m}\brk{\manifold{N}, b} \) is defined as
	\[
		\area{\sigma} = \inf\set{\area{f} \st \text{\( f \colon \overline{\Bset^{m}} \to \manifold{N} \) is Lipschitz, \( f \in \sigma \)}}\eqpunct{.}
	\]
\end{definition}

We are now in position to state the following result.

\begin{proposition}
\label{proposition_Etop_area}
	Assume that \( m \leq \dim{\manifold{N}} \).
	For every \( \sigma \in \pi_{m} \brk{\manifold{N},b} \),
	\[
		\energy^{1, m}_{\mathrm{top}} \brk{\sigma}
		\geq
		m^{m/2}\area{\sigma}\eqpunct{,}
	\]
	with equality if \( m = 2 \).
\end{proposition}

In the case where \( \manifold{N} \) is of dimension \( m \) and \( \sigma \) also corresponds to a \emph{homology cycle} \( \cycle{\sigma} \) -- which will be the case whenever the \emph{Hurewicz homomorphism} is an isomorphism -- then \( \area{\sigma} \) corresponds exactly to the area of the unique simplicial complex in the class of \( \cycle{\sigma} \), computed by summing the area of all the simplexes that it contains, counted with their multiplicity.

\begin{proof}[Proof of \cref{proposition_Etop_area}]
	The proof of the lower bound follows from the exact same argument as above, combining the arithmetico-geometric inequality with the area formula.
	We note that we may reduce to Lipschitz -- and even smooth -- maps in the computation of \( \energy^{1, m}_{\mathrm{top}} \brk{\sigma} \) by a straightforward approximation argument.
	
	To deduce the upper bound when \( m = 2 \), for every \( \varepsilon > 0 \), we apply Morrey's \( \varepsilon \)-conformality theorem~\cite{Morrey_1948} (see also~\cite{Fitzi_Wenger_2020}) to obtain a map \( f \in \sigma \) such that
	\[
		\int_{\Bset^{m}} \abs{\Deriv f}^{2}
		\leq
		2\area{\sigma} + \varepsilon\eqpunct{.}
	\]
	The conclusion follows by letting \( \varepsilon \to 0 \).
\end{proof}

What the proof actually shows is that it always holds that
\[
	\energy^{1, m}_{\mathrm{top}} \brk{\sigma}
	\geq
	m^{m/2}\inf\set[\bigg]{\int_{\Bset^{m}} \jac f \st f \in \sigma}
\]
with equality if \( m = 2\), hence relating the minimal Jacobian integral needed to realise a homotopy class with its topological energy, without assuming that \( m \leq \dim{\manifold{N}} \).
This Jacobian integral is in turn related to the area of \( f \) under the assumption \( m \leq \dim{\manifold{N}} \) via the area formula.
If on the contrary \( m > \dim{\manifold{N}} \), we instead have the relation
\[
	\int_{\Bset^{m}} \jac f
	=
	\int_{\manifold{N}} \mathcal{H}^{m-\dim{\manifold{N}}}\brk{f^{-1}\brk{\set{y}}}\dif y\eqpunct{,}
\]
relating the Jacobian integral with the \emph{coarea} of the map, which is less directly related to the intuitive idea of the area of the homotopy class represented by \( f \).

In \cref{proposition_Etop_sphere} and \cref{proposition_Etop_projective}, the topological energy grows linearly with respect to the relevant degree.
However, such a rate of growth is not universal: there are situations where the growth is governed by a power law with exponent less than \( 1 \).
The model case is the Hopf degree, which was first studied by Rivière~\cite{Riviere_1998} (see also~\citelist{\cite{Schikorra_VanSchaftingen_2020}\cite{Grzela_Mazowiecka}}).
We note however that the energy growth is always sublinear; see~\eqref{eq_thuav3woo1thie4eeNgeiNei} below.
In our framework, we have the following two-sided estimate for the topological energy with respect to the Hopf degree.

\begin{proposition}
\label{proposition_Etop_Hopf}
For every \( \sigma \in \pi_{4n - 1} \brk{\Sset^{2n},b} \),
	\[
		c \abs{\deg_{\mathrm{H}} \sigma}^{1 - \frac{1}{4n}}\le \energy^{1, 4n - 1}_{\mathrm{top}} \brk{\sigma}
		\le C \abs{\deg_{\mathrm{H}} \sigma}^{1 - \frac{1}{4n}}\eqpunct{,}
	\]
	where \(\deg_{\mathrm{H}}\) is the Hopf invariant.
\end{proposition}
\begin{proof}
When \(n = 1\), this follows from Rivière's sharp estimate on the Hopf invariant \cite{Riviere_1998};
when \(n \ge 2\), Rivière’s argument adapts straightforwardly to get the upper bound and a lower bound can be obtained thanks to Whitehead products \cite{Schikorra_VanSchaftingen_2020}.
\end{proof}

A similar phenomenon occurs for more involved homotopical quantities, connected to rational homotopy theory; we refer the reader to~\citelist{\cite{Hardt_Riviere_2008}\cite{Park_Schikorra_2023}} for more details.

After this review of model computations of \( \energy^{1, m}_{\mathrm{top}} \), and before we move to the definition of the disparity energy, we collect some fundamental properties of the topological energy.

\begin{proposition}[Norm properties of the topological energy]
\label{proposition_properties_Etop}
The quantity \(\energy^{1, m}_{\mathrm{top}}\) has  the following properties:
 \begin{enumerate}[label=(\roman*)]
  \item 
  \label{it_Etop_nonnegative}
  For every \(\sigma \in \pi_{m} \brk{\manifold{N}, b}\), one has \[
  \energy^{1, m}_{\mathrm{top}} \brk{\sigma} \ge 0\eqpunct{.}
  \]

  \item 
  \label{it_Etop_gap0}
  There exists \(\eta \in \intvo{0}{\infty}\) such that if \(\sigma \in \pi_{m} \brk{\manifold{N}, b}\) and 
  \[
  \energy^{1, m}_{\mathrm{top}} \brk{\sigma} < \eta\eqpunct{,}
  \]
  then \(\sigma = 0\).
\item 
\label{it_Etop_symmetric}
For every \(\sigma \in \pi_{m} \brk{\manifold{N}, b}\), one has 
\[
 \energy^{1, m}_{\mathrm{top}} \brk{-\sigma}
 = \energy^{1, m}_{\mathrm{top}} \brk{\sigma}\eqpunct{.}
\]
\item 
\label{it_Etop_sublinear}
For every \(\sigma, \tau \in \pi_{m} \brk{\manifold{N}, b}\), one has 
\[
 \energy^{1, m}_{\mathrm{top}} \brk{\sigma + \tau}
 \le \energy^{1, m}_{\mathrm{top}} \brk{\sigma}
 +
 \energy^{1, m}_{\mathrm{top}} \brk{\tau}\eqpunct{.}
\]
\item
\label{it_Etop_action} If \(\zeta \in \continuous \brk{\intvc{0}{1}, \manifold{N}}\) and \(\sigma \in \pi_m \brk{\manifold{N}, \zeta \brk{1}}\), then
\[
 \energy^{1, m}_{\mathrm{top}} \brk{\zeta_* \sigma}
 =\energy^{1, m}_{\mathrm{top}} \brk{\sigma}\eqpunct{.}
\]
 \item \label{it_Etop_discrete} 
 The set 
 \[
  \set{\energy^{1, m}_{\mathrm{top}} \brk{\sigma} \st \sigma \in \pi_{m}\brk{\manifold{N}, b}}
 \]
 is discrete.
 \end{enumerate}
\end{proposition}

In other words, \(\energy^{1, m}_{\mathrm{top}}\) is a norm on the group \(\pi_m \brk{\manifold{N}, b}\) for which the isomorphisms induced by the action of the fundamental groupoid \(\Pi_1\brk{\manifold{N}}\) are isometries.

It follows from \cref{proposition_properties_Etop} \ref{it_Etop_symmetric} and \ref{it_Etop_sublinear} that \begin{equation}
\label{eq_thuav3woo1thie4eeNgeiNei}
  \energy^{1, m}_{\mathrm{top}} \brk{k\sigma}          
  \le \abs{k} \, \energy^{1, m}_{\mathrm{top}} \brk{\sigma}\eqpunct{.}                   
\end{equation}
An important phenomenon to mention is that, in general, the reverse inequality to \eqref{eq_thuav3woo1thie4eeNgeiNei} does not hold; this is the case for example when \(m = 4n - 1\) and \(\manifold{N} = \Sset^{2n}\), one has
\( \smash{\energy^{1, 4n - 1}_{\mathrm{top}}} \brk{k\sigma}\simeq \abs{k}^{1 - 1/(4n)}\) (\cref{proposition_Etop_Hopf}).
The question of determining the optimal growth of \( \smash{\energy^{1, m}_{\mathrm{top}}} \brk{k\sigma} \) with respect to \( k \) (which is \( \abs{k} \) for the Brouwer degree by virtue of~\cref{proposition_Etop_sphere} and \(\abs{k}^{1 - 1/(4n)}\) for the Hopf degree according to \cref{proposition_Etop_Hopf}) is a challenging problem, which remains open in full generality; see for instance~\cite{Hardt_Riviere_2008} and the references therein for more details and other partial results.

In particular, \ref{it_Etop_symmetric} implies that the quantity 
\[
   \energy^{1, m}_{\mathrm{top}} \brk{\sqb{u, v, B_{\rho} \brk{a}}}
\]
is well-defined and \ref{it_Etop_gap0} that it vanishes if and only if \(u\) and \(v\) are homotopic relatively to \(\partial B_\rho \brk{a}\).
It follows from \cref{lemma_superpose_disparities} and \cref{proposition_properties_Etop} \ref{it_Etop_sublinear} and~\ref{it_Etop_action} that 
\[
   \energy^{1, m}_{\mathrm{top}} \brk{\sqb{u, w, B_{\rho} \brk{a}}}
  \le \energy^{1, m}_{\mathrm{top}} \brk{\sqb{u, v, B_{\rho} \brk{a}}}
  + \energy^{1, m}_{\mathrm{top}} \brk{\sqb{v, w, B_{\rho} \brk{a}}}\eqpunct{.}
\]

The assertion \ref{it_Etop_action} implies that 
\[
\energy^{1, m}_{\mathrm{top}}\brk[\big]{
  \brk{u \compose \gamma}_* \sqb{u, v, B_\rho \brk{\gamma \brk{1}}}}
  = \energy^{1, m}_{\mathrm{top}}\brk[\big]{\sqb{u, v, B_\rho \brk{\gamma \brk{0}}}}\eqpunct{.}
\]

\begin{proof}[Proof of \cref{proposition_properties_Etop}]
The nonnegativity in \ref{it_Etop_nonnegative} follows from the definition.

By \cref{proposition_stability_homotopy_classes} and the Poincaré inequality, there exists \(\eta \in \intvo{0}{\infty}\) such that, if 
\(f \in \sobolev^{1, m} \brk{\Bset^m, \manifold{M}}\), \(\tr_{\partial \Bset^m} f = b\), and 
\[
 \int_{\Bset^m} \abs{\Deriv f }^m < \eta\eqpunct{,}
\]
then \(f\) is homotopic to constant.
Thus, if \(f \in \sigma\), then \(\sigma = 0\), and \ref{it_Etop_gap0} follows.

For \ref{it_Etop_symmetric}, if \(\Phi \colon \Bset^m \to \Bset^m\) is an orientation reversing isometry, given \(f \in \sobolev^{1, m} \brk{\Bset^m, \manifold{N}}\) and \(f \in \sigma\), we have \(f \compose \Phi \in -\sigma\) and 
\[
  \int_{\Bset^m} \abs{\Deriv \brk{f \compose \Phi}}^m
  = \int_{\Bset^m} \abs{\Deriv f }^m\eqpunct{,}
\]
so that \(\energy^{1, m}_{\mathrm{top}} \brk{-\sigma}
 \le \energy^{1, m}_{\mathrm{top}} \brk{\sigma}\).

For \ref{it_Etop_sublinear}, assume that \(f, g \in \sobolev^{1, m}\brk{\Bset^m, \manifold{N}}\) and that \(f \in \sigma\) and \(g \in \tau\).
Taking \(c, d \in \Bset^m\) and \(\delta \in \intvo{0}{1}\) such that \(B_\delta \brk{c}\) and \(B_\delta \brk{d}\) are disjoint subsets of \(\Bset^m\) and defining 
\[
 h \brk{x}
 \defeq 
 \begin{cases}
  f \brk{\frac{x - c}{\delta}} & \text{if \(x \in B_\delta \brk{c}\),}\\
  g \brk{\frac{x - d}{\delta}} & \text{if \(x \in B_\delta \brk{d}\),}\\
  b & \text{otherwise,}
 \end{cases}
\]
we have \(h \in \sigma + \tau\) and 
\[
 \int_{\Bset^m} \abs{\Deriv h}^m
 = \int_{\Bset^m} \abs{\Deriv f}^m + \int_{\Bset^m} \abs{\Deriv g}^m\eqpunct{.}
\]
The conclusion then follows from \eqref{eq_of7vae7eiquooDooh8nohnai}.

In order to prove \ref{it_Etop_action}, without loss of generality we can assume that \(\zeta \in \continuous^1 \brk{\intvc{0}{1}, \manifold{N}}\).
Given \(f \in \sobolev^{1, m}\brk{\Bset^m, \manifold{N}}\) such that \(f \in \sigma\), we define for every \(\varepsilon \in \intvo{0}{\infty}\)
\[
  f_\varepsilon =
  \begin{cases}
    \zeta \brk{2 - 2\abs{x}^\varepsilon} & \text{if \(\abs{x}^\varepsilon \ge 1/2\),}\\
    f \brk{2^{1/\varepsilon} x} &\text{otherwise.}
  \end{cases}
\]
We compute then 
\[
\begin{split}
   \int_{\Bset^m} \abs{\Deriv f_\varepsilon}^m
   &=  \int_{\Bset^m} \abs{\Deriv f}^m + 
   \int_{\abs{x}^\varepsilon \ge 1/2} \frac{ \varepsilon^m  \abs{\zeta' \brk{2 - 2\abs{x}^\varepsilon}}^m}{\abs{x}^{m \brk{1 - \varepsilon}}} \dif x\\
   &\le \int_{\Bset^m} \abs{\Deriv f}^m + \C \varepsilon^{m - 1}\eqpunct{,}
  \end{split}
\]
proving thus that
\[
 \energy^{1, m}_{\mathrm{top}} \brk{\zeta_* \sigma}
 \le \energy^{1, m}_{\mathrm{top}} \brk{\sigma}\eqpunct{.}
\]
Considering the reverse path \(\check{\zeta}\) defined by \(\check{\zeta}\brk{t} = \zeta \brk{1 - t}\), we get, by homotopy invariance,
\[
 \energy^{1, m}_{\mathrm{top}} \brk{\sigma}
 = \energy^{1, m}_{\mathrm{top}} \brk{\check{\zeta}_* \zeta_* \sigma}
 \le \energy^{1, m}_{\mathrm{top}} \brk{\zeta_* \sigma}\eqpunct{,}
\]
and the conclusion follows.

Finally, \ref{it_Etop_discrete} follows from the decomposition into a bounded number of maps taken in a finite set that are glued together through the action of \(\Pi_{1}\brk{\manifold{N}}\) \citelist{\cite{Duzaar_Kuwert_1998}*{th.\ 4}\cite{Schoen_Wolfson_2001}*{lem.\ 5.2}\cite{VanSchaftingen_2020}} and from \ref{it_Etop_action}.
\resetconstant
\end{proof}

\begin{proposition}
\label{proposition_estimate_E1mtop_u_v}
Let \(u, v\in \sobolev^{1, m} \brk{\Bset^m, \manifold{N}}\cap \continuous \brk{\overline{\Bset^m}, \manifold{N}}\).
If \(u = v\) on \(\partial \Bset^m\), then 
\[
 \energy^{1, m}_{\mathrm{top}}
 \brk{\sqb{u, v, \Bset^m}}
 \le \int_{\Bset^m}\abs{\Deriv u}^m + \int_{\Bset^m} \abs{\Deriv v}^m\eqpunct{.}
\]
\end{proposition}
\begin{proof}
By an approximation argument, we can assume that \(u = b\) on \(\Bar{B}_\rho\).
We define then 
\[
 w \brk{x}
 \defeq 
 \begin{cases}
  v \brk{\frac{x}{\rho}} & \text{if \(\abs{x}\le \rho\),}\\
  u \brk{\frac{\rho x}{\abs{x}^2}} &\text{if \(\abs{x} \ge \rho\).}
 \end{cases}
\]
Since \(v = u\) on \(\partial \Bset^m\), we have \(w \in \sobolev^{1, m} \brk{\Bset^m, \manifold{N}} \cap \continuous \brk{\overline{\Bset^m}, \manifold{N}}\), while it is straightforward to observe that \( w \in \sqb{u, v, \Bset^m} \).
By our additional assumption on \(u\), \(w\restr{\partial \Bset^m} = b\).
Since the change of variable in the definition of \(w\) is conformal, we have 
\[
 \int_{\Bset^m}\abs{\Deriv w}^m
 = \int_{\Bset^m}\abs{\Deriv u}^m + \int_{\Bset^m} \abs{\Deriv v}^m\eqpunct{,}
\]
and the conclusion follows.
\end{proof}

\subsection{The disparity energy}
\label{section_disparity_energy}
We are finally in position to define the energy associated with topological disparities.
We define the \emph{disparity energy}
\begin{multline}
\label{eq_idei2jeej5luop9AiB8aeph0}
 \energy^{1, m}_{\mathrm{disp}} \brk{u, v}
 \defeq 
 \inf \Bigl\{\energy^{1, m}_{\mathrm{top}} \brk{\sqb{u, w, B_{\rho} \brk{a}}} \stSymbol[\Big] w \in \continuous \brk{\manifold{M}, \manifold{N}} \text{ homotopic to \(v\) }\\
 \text{ and } u = w \text{ in \(\manifold{M} \setminus B_\rho\brk{a}\)}\Bigr\}\eqpunct{.}
\end{multline}

The infimum above runs over all balls \( B_{\rho} \brk{a} \subset \manifold{M} \) with \( \rho \) sufficiently small, and all corresponding maps \( w \).
However,~\cref{lemma_moving_disparities} and \cref{proposition_properties_Etop}~\ref{it_Etop_action} imply that we could as well have fixed the point \( a \), and that the quantity \(
  \energy^{1, m}_{\mathrm{disp}} \brk{u, v}
\) would not depend on this choice.
By \cref{proposition_properties_Etop} \ref{it_Etop_discrete}, the infimum in \eqref{eq_idei2jeej5luop9AiB8aeph0} is actually a minimum.

\begin{proposition}
\label{proposition_disparity_distance}
The function \(\energy^{1, m}_{\mathrm{disp}}\) is a distance on homotopy classes:
\begin{enumerate}[label=(\roman*)]
 \item 
 \label{it_Edis_nonnegative}
 For every \(u, v \in \continuous \brk{\manifold{M}, \manifold{N}}\), 
 \[
  \energy^{1, m}_{\mathrm{disp}} \brk{u, v} \ge 0\eqpunct{.}
 \]
 \item \label{it_Edis_discrete}
 There exists \(\eta \in \intvo{0}{\infty}\) such that if \(u, v \in \continuous \brk{\manifold{M}, \manifold{N}}\) and 
 \[
   \energy^{1, m}_{\mathrm{disp}} \brk{u, v} \le \eta\eqpunct{,}
 \]
 then \(u\) and \(v\) are homotopic.
 \item 
 \label{it_Edis_symmetric}
 For every \(u, v \in \continuous \brk{\manifold{M}, \manifold{N}}\), 
 \[
   \energy^{1, m}_{\mathrm{disp}} \brk{u, v}
   = \energy^{1, m}_{\mathrm{disp}} \brk{v, u}\eqpunct{.}
 \]
 \item 
  \label{it_Edis_triangle}
For every \(u, v, w \in \continuous \brk{\manifold{M}, \manifold{N}}\), we have
\[
 \energy^{1, m}_{\mathrm{disp}} \brk{u, w}
 \le \energy^{1, m}_{\mathrm{disp}} \brk{u, v} + \energy^{1, m}_{\mathrm{disp}} \brk{v, w}\eqpunct{.}
\]
\end{enumerate}
\end{proposition}
\begin{proof}
Given \(u, v \in \continuous \brk{\manifold{M}, \manifold{N}}\), the quantity \(\energy^{1, m}_{\mathrm{disp}} \brk{u, v}\) is clearly nonnegative, so that \ref{it_Edis_nonnegative} holds.

If moreover \(\energy^{1, m}_{\mathrm{disp}} \brk{u, v} < \eta\), then there exists \(w \in \continuous \brk{\manifold{M}, \manifold{N}}\) such that \(u = w\) in \(\manifold{M} \setminus B_\rho \brk{a}\), \(v\) and \(w\) are homotopic, and 
\[
 \energy^{1, m}_{\mathrm{top}} \brk{\sqb{u, w, B_{\rho} \brk{a}}} <\eta\eqpunct{.}
\]
By \cref{proposition_properties_Etop} \ref{it_Etop_gap0} and \cref{lemma_disparity_trivial_homotopic}, \(u\) and \(w\) are homotopic relatively to \( \manifold{M} \setminus B_\rho \brk{a}\), proving \ref{it_Edis_discrete}.

Let \(u, v \in \continuous \brk{\manifold{M}, \manifold{N}}\). Assuming that \(\Tilde{v} \in \continuous \brk{\manifold{M}, \manifold{N}}\) is homotopic to \(v\) and \(\Tilde{v} = u\) in \(\manifold{M} \setminus B_{\rho}\brk{a}\),
we let \(\Tilde{u} \in \continuous \brk{\manifold{M}, \manifold{N}}\) and \(\zeta \in \continuous \brk{\intvc{0}{1}, \manifold{N}}\) be given by \cref{lemma_discrepancy_homotopy_transport}.
We compute, with the aid of~\cref{proposition_properties_Etop} \ref{it_Etop_action},
\[
\begin{split}
 \energy^{1, m}_{\mathrm{top}} \brk{\sqb{\Tilde{v}, u, B_{\rho} \brk{a}}}
 &= \energy^{1, m}_{\mathrm{top}} \brk{\zeta_* \sqb{v, \Tilde{u}, B_{\rho} \brk{a}}}\\
 &= \energy^{1, m}_{\mathrm{top}} \brk{\sqb{v, \Tilde{u}, B_{\rho} \brk{a}}}\eqpunct{.}
\end{split}
\]
To conclude, we let \(\gamma \in \continuous \brk{\intvc{0}{1}, \manifold{M}}\) be such that \(\gamma\brk{1} = a\) and \(\gamma \brk{0}\not \in B_\rho \brk{a}\), and we rely on~\cref{proposition_properties_Etop} \ref{it_Etop_action}  and~\cref{lemma_disparity_reciprocal} to obtain
\[
\begin{split}
	\energy^{1, m}_{\mathrm{top}} \brk{\sqb{\Tilde{v}, u, B_{\rho} \brk{a}}}
	&= \energy^{1, m}_{\mathrm{top}} \brk{\brk{\Tilde{v} \compose \gamma}_* \sqb{\Tilde{v}, u, B_{\rho} \brk{a}}} \\
	&= \energy^{1, m}_{\mathrm{top}} \brk{-\brk{u \compose \gamma}_* \sqb{u, \Tilde{v}, B_{\rho} \brk{a}}}= \energy^{1, m}_{\mathrm{top}} \brk{\sqb{u, \Tilde{v}, B_{\rho} \brk{a}}}\eqpunct{,}
\end{split}
\]
so that \ref{it_Edis_symmetric} follows using~\cref{proposition_properties_Etop} \ref{it_Etop_symmetric}.

If \(u, v, w \in \continuous \brk{\manifold{M}, \manifold{N}}\),
assume that \(\Tilde{v}\in \continuous \brk{\manifold{M}, \manifold{N}}\) is homotopic to \(v\) and \(\Tilde{v} = u\) on \(\manifold{M}\setminus B_\rho\brk{a}\), and that \(\Tilde{w}\in \continuous \brk{\manifold{M}, \manifold{N}}\) is homotopic to \(w\) and \(\Tilde{w} = v\) on \(\manifold{M}\setminus B_\rho\brk{a}\).
By \cref{lemma_discrepancy_homotopy_transport}, there exist \(\smash{\Tilde{\Tilde{w}}} \in \continuous \brk{\manifold{M}, \manifold{N}}\) and \(\zeta \in \continuous \brk{\intvc{0}{1}, \manifold{N}}\) such that \(\smash{\Tilde{\Tilde{w}}}\) is homotopic to \(\Tilde{w}\) and \(w\), \(\smash{\Tilde{\Tilde{w}}} = \Tilde{v} = u\) in \(\manifold{M} \setminus B_{\rho}\brk{a}\), and
\begin{align*}
 \sqb{v, \Tilde{w}, B_{\rho}\brk{a}}
 &= \zeta_* \sqb{\Tilde{v}, \Tilde{\Tilde{w}}, B_{\rho}\brk{a}}&
 &\text{in } \pi_m \brk{\manifold{N}, v \brk{a}}\eqpunct{.}
\end{align*}
Fixing \(\gamma \in \continuous \brk{\intvc{0}{1}, \manifold{M}}\) such that \(\gamma\brk{1} = a\) and \(\gamma \brk{0}\not \in B_\rho \brk{a}\), by \cref{proposition_properties_Etop} \ref{it_Etop_action}, \cref{lemma_superpose_disparities}, and \cref{proposition_properties_Etop} \ref{it_Etop_sublinear}, we have 
\[
\begin{split}
 \energy^{1, m}_{\mathrm{top}} \brk{\sqb{u, \Tilde{\Tilde{w}}, B_{\rho} \brk{a}}}
 &= \energy^{1, m}_{\mathrm{top}} \brk{\brk{u \compose \gamma}_*\sqb{u, \Tilde{\Tilde{w}}, B_{\rho} \brk{a}}}\\
 &= \energy^{1, m}_{\mathrm{top}} \brk{\brk{u \compose \gamma}_*\sqb{u, \Tilde{v}, B_{\rho} \brk{a}} + \brk{\Tilde{v} \compose \gamma}_*\sqb{\Tilde{v}, \Tilde{\Tilde{w}}, B_{\rho} \brk{a}}}\\
 &\le \energy^{1, m}_{\mathrm{top}} \brk{\brk{u \compose \gamma}_*\sqb{u, \Tilde{v}, B_{\rho} \brk{a}}} + \energy^{1, m}_{\mathrm{top}} \brk{\brk{\Tilde{v} \compose \gamma}_*\sqb{\Tilde{v}, \Tilde{\Tilde{w}}, B_{\rho} \brk{a}}}\\
 &= \energy^{1, m}_{\mathrm{top}} \brk{\brk{u \compose \gamma}_*\sqb{u, \Tilde{v}, B_{\rho} \brk{a}}} + \energy^{1, m}_{\mathrm{top}} \brk{\zeta_*\sqb{\Tilde{v}, \Tilde{\Tilde{w}}, B_{\rho} \brk{a}}}\\
 &= \energy^{1, m}_{\mathrm{top}} \brk{\sqb{u, \Tilde{v}, B_{\rho} \brk{a}}} + \energy^{1, m}_{\mathrm{top}} \brk{\sqb{v, \Tilde{w}, B_{\rho} \brk{a}}}\eqpunct{,}
\end{split}
\]
so that \ref{it_Edis_triangle} follows. 
\end{proof}

Another important property of the disparity energy is its continuity with respect to the strong \( \sobolev^{1,m} \) convergence.

\begin{proposition}
\label{proposition_continuity_disparity}
	If \( u \in \sobolev^{1,m} \brk{\manifold{M}, \manifold{N}} \cap \continuous \brk{\manifold{M}, \manifold{N}} \), \( v \in \continuous \brk{\manifold{M}, \manifold{N}} \), and \( \brk{u_j}_{j \in \Nset} \) is a sequence in \( \sobolev^{1,m} \brk{\manifold{M}, \manifold{N}} \cap \continuous \brk{\manifold{M}, \manifold{N}} \) converging strongly to \( u \) in \( \sobolev^{1,m} \), then
	\[
		\lim_{j \to \infty} \energy^{1,m}_{\mathrm{disp}} \brk{u_j,v} = \energy^{1,m}_{\mathrm{disp}} \brk{u,v}\eqpunct{.}
	\]
\end{proposition}
\begin{proof}
	This follows from \cref{proposition_stability_homotopy_classes} and \cref{proposition_disparity_distance}~\ref{it_Edis_symmetric}.
\end{proof}

At the heart of our definition of the disparity energy is the following problem: given two maps \( u \) and \( v \), find the optimal way of modifying \( u \) inside one ball to obtain a map homotopic to \( v \), in order to minimize the energy of the resulting map.
The next proposition essentially encodes the fact that one cannot gain by allowing instead to modify \( u \) on \emph{several balls}.

\begin{proposition}
\label{proposition_Edis_controlled_several_bubbles}
If \(u = v \) on \(\manifold{M} \setminus \bigcup_{i = 1}^{I} B_{\rho} \brk{a_i}\) with \(B_\rho \brk{a_1}, \dotsc, B_\rho\brk{a_I}\) disjoint, then
\[
 \energy^{1, m}_{\mathrm{disp}} \brk{u, v}
 \le \sum_{i = 1}^{I} \energy^{1, m}_{\mathrm{top}}\brk[\big]{\sqb{u, v, B_\rho \brk{a_i}}}\eqpunct{.}
\]
\end{proposition}

The idea of the proof is to use the merging tool provided by~\cref{label_merging} to gather all the bubbles contained in the balls \( B_{\rho}\brk{a_i} \) inside a common ball \( B_{\rho}\brk{a} \).

\begin{proof}[Proof of \cref{proposition_Edis_controlled_several_bubbles}]
 Choosing \(\gamma_i \in \continuous \brk{\intvc{0}{1}, \manifold{M}}\) such that \(\gamma_i \brk{0} = a \) and \(\gamma_i \brk{1} = a_i \), with \( B_{\rho}\brk{a} \cap \bigcup_{i = 1}^{I} B_{\rho} \brk{a_i} = \varnothing \), and applying \cref{label_merging}, we get a mapping \(w \in \continuous \brk{\manifold{M}, \manifold{N}}\) such that  
\(u = w \) in \(\manifold{M} \setminus B_{\rho}\brk{a}\), \( w \) is homotopic to \( v \), and 
\begin{align*}
 \sqb{u, w, B_\rho\brk{a}}
 &= \sum_{i = 1}^I \brk{u \compose \gamma_i}_*
 \sqb{u, v, B_\rho\brk{a_i}}&
 &\text{in } \pi_m \brk{\manifold{N}, u \brk{a}}\eqpunct{.}
\end{align*}
Hence, by \cref{proposition_properties_Etop} \ref{it_Etop_sublinear} and \ref{it_Etop_action}, we get  
\[
\begin{split}
  \energy^{1, m}_{\mathrm{top}} \brk{\sqb{u, w, B_\rho\brk{a}}}
   &\le \sum_{i = 1}^I \energy^{1, m}_{\mathrm{top}} \brk{\brk{u \compose \gamma_i}_*\sqb{u, v, B_\rho\brk{a_i}}}\\
   &= \sum_{i = 1}^I \energy^{1, m}_{\mathrm{top}} \brk{\sqb{u, v, B_\rho\brk{a_i}}}\eqpunct{,}
\end{split}
\]
and the conclusion follows.
\end{proof}

At the core of the proof of \cref{proposition_Edis_controlled_several_bubbles} lies the estimate
\[
 \energy^{1, m}_{\mathrm{disp}} \brk{u, v}
 \le 
 \energy^{1, m}_{\mathrm{top}}\brk[\Big]{\sum_{i = 1}^I \brk{u \compose \gamma_i}_* \sqb{u, v, B_\rho \brk{a_i}}}\eqpunct{.}
\]
When \(\pi_1 \brk{\manifold{N}}\simeq \set{0}\), then the quantity on the right-hand side is independent on the choice of the paths \(\gamma_1, \dotsc, \gamma_k\).

In general however, it depends on \(\gamma_1, \dotsc, \gamma_k\).
For example, if we take the projective space \(\manifold{M} = \RPset^{2n}\) and the sphere \(\manifold{N} = \Sset^{2n}\), if \(u = v\) on \(\manifold{M} \setminus \bigcup_{i = 1}^{2} B_{\rho} \brk{a_i}\)
and \(\sqb{u, v, B_{\rho}\brk{a_i}}\) is a map of Brouwer degree \(d_i\), then depending on the way \(\gamma_1\) and \(\gamma_2\) transport the orientation, either \( \brk{u \compose \gamma_1}_*\sqb{u, v, B_{\rho}\brk{a_1}} + \brk{u \compose \gamma_2}_*\sqb{u, v, B_{\rho}\brk{a_2}}\) will have degree \(\pm d_1 \pm d_2\), leading to different values of the energy.

Even when \(\manifold{M}\) is orientable, a non-trivial action of \(\pi_{1}\brk{\manifold{M}}\) can also create such phenomena. 
Indeed, if \(\manifold{M} = \Sset^{m - 1}\times \Sset^1\), if we take \(u, v\) such that \(\zeta = u \compose \Bar{\gamma}\) is nontrivial for \(\Bar{\gamma}\) a loop along \(\Sset^1\), we see that we have to consider (at least) all the quantities 
\[
 \energy^{1, m}_{\mathrm{disp}} \brk{u, v}
 \le 
 \energy^{1, m}_{\mathrm{top}}\brk[\Big]{\sum_{i = 1}^I \brk{u \compose \gamma_i}_* \zeta_*^{k_i}\sqb{u, v, B_\rho \brk{a_i}}}\eqpunct{,}
\]
for \(k_i \in \Zset\), that have no a priori reason of being equal.

The process of changing \(\gamma_i\) only revolves around mappings that are homotopic to \(v\) relatively to \(\manifold{M}^{m - 2}\) (we can assume that the paths and the balls never intersect \(\manifold{M}^{m - 2}\)); it turns out that maps can be freely homotopic without being relatively homotopic.

\begin{proposition}
\label{proposition_free_vs_relative_homotopy_RPn}
Let \(n \in \Nset \setminus \set{0}\).
\begin{enumerate}[label=(\roman*)]
 \item If \(n\) is odd, if \(f, g \in \continuous \brk{\RPset^n, \RPset^n}\) are freely homotopic, and if \(f \brk{a} = g \brk{a}\), then \(f\) and \(g\) are homotopic relatively to \(\set{a}\).
 
 \item If  \(n\) is even and \(a \in \RPset^{n - 1}\subseteq \RPset^n\), then there exist \(f, g \in \continuous \brk{\RPset^n, \RPset^n}\) such that \(f= g\) on \(\RPset^{n - 1}\) and \(f\) and \(g\) are freely homotopic, but \(f\) and \(g\) are not homotopic relatively to \(\set{a}\).
\end{enumerate}
\end{proposition}

As a consequence of \cref{proposition_free_vs_relative_homotopy_RPn}, if
\(\manifold{N} = \RPset^n\) with \(n \in \Nset \setminus \set{0}\) even
and \(\manifold{M} = \RPset^n \times \manifold{M}''\) with \(\dim \manifold{M}'' = m - n\), then there are maps \(f, g \in \continuous \brk{\manifold{M}, \manifold{N}}\)
such that \(f\) and \(g\) are freely homotopic, \(f\restr{\manifold{M}^{m - 1}} = g \restr{\manifold{M}^{m - 1}}\), but \(f\) and \(g\) are not homotopic relatively to \(\manifold{M}^{m - n}\).
Indeed, the homotopy extension property shows that the maps in \cref{proposition_free_vs_relative_homotopy_RPn} can be taken to be equal outside an arbitrarily small ball of \(\RPset^n\) so that they coincide of the \(\brk{n-1}\)-component of its triangulation.

\begin{proof}[Proof of \cref{proposition_free_vs_relative_homotopy_RPn}]
We let \(\pi \colon \Sset^n \to \RPset^n\) be the universal covering of the projective space \(\RPset^n\) by \(\Sset^n\). 
Under the embedding as rank-one projections \(\RPset^n \subseteq \Rset^{\brk{n + 1}\times \brk{n + 1}}\), we have \(\pi \brk{x} = x \otimes x\).

If \(n = 2\ell + 1\) is odd, assume that \(f = H \brk{0, \cdot}\) and \(g = H \brk{1, \cdot}\) for some \(H \in \continuous \brk{\intvc{0}{1} \times \RPset^n, \RPset^n}\). If \(H \brk{\cdot, a}\) is homotopic to a constant, then an application of the homotopy extension property gives a homotopy relative to \(a\).
Otherwise, one can note that if we define \(\lifting{G} \in \continuous \brk{\intvc{0}{1}\times \Sset^n, \Sset^n}\) for \(x = \brk{x', x''} \in \Sset^{n} \subseteq \Rset^{\ell + 1} \times \Rset^{\ell + 1}\) by
\[
 \lifting{G} \brk{t, x} = \brk{x' \cos \brk{\pi t} - x'' \sin \brk{\pi t},x' \sin \brk{\pi t} + x'' \cos \brk{\pi t}}\eqpunct{,}
\]
and then, since \(\Tilde{G} \brk{t, -x} = - \Tilde{G}\brk{t, x}\), \(G \in \continuous \brk{\intvc{0}{1}\times \RPset^n, \RPset^n}\) for every \(x \in \Sset^n\) by
\[
  G \brk{t, \pi \brk{x}} = \pi \brk{\lifting{G} \brk{t,\brk{x}}}\eqpunct{,}
\]
we see that \(G \brk{0, \cdot} = G \brk{1, \cdot} = \operatorname{id}_{\RPset^n}\) while for every \(x \in \RPset^n\), \(G \brk{\cdot, x}\) is a generator of \(\pi_1 \brk{\RPset^n}\). Hence, \(F \brk{t, x} \defeq G \brk{t, f \brk{x}}\) defines a homotopy between \(f\) and itself such that the homotopy class of \(F \brk{\cdot, a} = G \brk{\cdot, f \brk{a}}\)  is a generator of \(\pi_1 \brk{\RPset^n}\) (and the unique nontrivial element when \(n \ge 2\)). By combining suitable homotopies, this brings us back to the case where \(H \brk{\cdot, a}\) is homotopic to a constant.

When \(n\) is even, we define
\(f \defeq \operatorname{id}_{\RPset^n}\). Taking \(\lifting{g} \brk{x} \defeq \brk{x', -x''}\) for \(\brk{x', x''} \in \Sset^n \subseteq \Rset^2 \times \Rset^{n - 1}\), we set, since \(\lifting{g} \brk{-x} = - \lifting{g} \brk{x}\),
\[
g \brk{\pi \brk{x}} = \pi \brk{\lifting{g} \brk{x}}\eqpunct{,}
\]
and we fix \(a = \pi \brk{\lifting{a}}\) with \(\lifting{a} = \brk{1, 0, \dotsc, 0}\).
Setting for \(t \in \intvc{0}{1}\) and \(x \in \Sset^n\)
\[
\lifting{K} \brk{t, x} = \brk{x_1 \cos \brk{\pi t} - x_2 \sin \brk{\pi t},x_1 \sin \brk{\pi t} + x_2 \cos \brk{\pi t}, x''}
\]
and, since \(\lifting{K}\brk{t, -x} = - \lifting{K}\brk{t, x}\),
\[
K \brk{t, \pi \brk{x}} = \pi \brk{\lifting{K} \brk{t,\brk{x}}}\eqpunct{,}
\]
we have \(K \brk{0, \cdot} = f\) and \(K \brk{1, \cdot} = g\), so that \(f\) and \(g\) are freely homotopic.
We assume now that there is some \(H \in \continuous \brk{\intvc{0}{1}\times \RPset^n, \RPset^n}\) such that \(H \brk{\cdot, a} = a\), \(H \brk{0, \cdot} = f\), and \(H \brk{1, \cdot} = g\). 
By the classical theory of lifting, since \(\Sset^n\) is simply connected for \(n \ge 2\), there exists \(\lifting{H} \in \continuous \brk{\intvc{0}{1}\times \Sset^n, \Sset^n}\) such that
for every \(\brk{t,x} \in \intvc{0}{1}\times \Sset^n\),
\[
\pi \brk{\lifting{H} \brk{t, x}} = H \brk{t, \pi \brk{x}}
\]
and \(\lifting{H} \brk{\cdot, a} = \lifting{a}\).
It follows then that 
\(\lifting{H}\brk{0, \cdot} = \operatorname{id}_{\Sset^n}\) and
\(\lifting{H}\brk{1, \cdot} = \lifting{g}\).
Since \(\deg \brk{\operatorname{id}_{\Sset^n}} = 1\) and \(\deg \brk{\lifting{g}} = \brk{-1}^{n - 1} = -1\), this cannot be the case.
\end{proof}

We now have at our disposal all the notions that are required to state and prove the upper bound on the heterotopic energy as stated in~\cref{theorem_intro_heterotopic_disparity}.

\begin{proposition}
\label{proposition_upper_estimate}
If \(u \in \sobolev^{1, m} \brk{\manifold{M}, \manifold{N}} \cap \continuous \brk{\manifold{M}, \manifold{N}}\) and \(v \in \continuous \brk{\manifold{M}, \manifold{N}}\), then
\[
 \energy^{1, m}_{\mathrm{het}}
  \brk{u, v} 
  \le \int_{\manifold{M}} \abs{\Deriv u}^m
  +  \energy^{1, m}_{\mathrm{disp}} \brk{u, v}\eqpunct{.}
\]
\end{proposition}

Our main tool for the proof of \cref{proposition_upper_estimate} is the following \emph{opening lemma} (see \cite{Brezis_Li_2001}*{Lemma 2.1}).

\begin{lemma}
\label{lemma_opening}
We define
\[
 u_r \brk{x}
 \defeq
 \begin{cases}
 u \brk{x - r \frac{x}{\abs{x}}} & \text{if \( \abs{x} \ge r \)}\eqpunct{,}\\
 u \brk{0} & \text{if \(\abs{x} \le r\)}\eqpunct{.}
 \end{cases}
\]
If
\[
 \int_{\Bset^m} \frac{\abs{\Deriv u \brk{x}}^m}{\abs{x}^{m - 1}} \dif x < \infty\eqpunct{,}
\]
then
\[
 \lim_{r \to 0} \int_{\Bset^m} \abs{\Deriv u_r}^m 
 = \int_{\Bset^m} \abs{\Deriv u}^m\eqpunct{.}
\]
\end{lemma}
\begin{proof}
Setting for \(r \in \intvo{0}{1}\) and \(x \in \Bset^m\), \(\Psi_r \brk{x} \defeq x + r x/\abs{x}\),
we have \(\Psi^{-1} \brk{x} = x - r x/\abs{x}\).
By the change of variable formula, we have
\[
 \int_{\Bset^m \setminus B_r} \abs{\Deriv \brk{u \compose \Psi_r^{-1}}}^m
 \le \int_{\Bset^m} \abs{\Deriv u}^m \abs{\brk{\Deriv \Psi_r}^{-1}}^m
  \jac \Psi_r
  \le \int_{\Bset^m}  \abs{\Deriv u \brk{x}}^m \brk[\bigg]{1 + \frac{r}{\abs{x}}}^{m - 1}
 \dif x\eqpunct{,}
\]
and the conclusion follows.
\end{proof}

\begin{proof}[Proof of \cref{proposition_upper_estimate}]
We first assume that \(u \in \continuous^1 \brk{\manifold{M}, \manifold{N}}\).
Given \( \varepsilon > 0 \), we fix \( a \in \manifold{M} \), \( \rho \in \intvo{0}{\infty} \), and \( w \in \continuous \brk{\manifold{M}, \manifold{N}} \) such that \( w \) is homotopic to \( v \), \( w = u \) in \( \manifold{M} \setminus B_{\rho} \brk{a} \), and 
\[
	\energy_{\mathrm{top}}^{1,m} \brk{\sqb{u,w,B_{\rho}\brk{a}}}
	\leq
	\energy^{1, m}_{\mathrm{disp}} \brk{u, v} + \varepsilon\eqpunct{.}
\]
By \cref{lemma_opening}, there exists \(u_{\varepsilon}\) constant on \( B_{\varepsilon} \brk{a} \) such that 
\(u_{\varepsilon}\) is homotopic to \(u\) relatively to \(\manifold{M}\setminus B_\rho\brk{a} \cup \set{a}\)
and 
\[
 \lim_{\varepsilon \to 0} \int_{\manifold{M}}\abs{\Deriv u_\varepsilon}^m
 \le \int_{\manifold{M}}\abs{\Deriv u}^m\eqpunct{.}
\]
By the homotopy properties, we have 
\begin{align*}
 \sqb{u_\varepsilon, w, B_\rho\brk{a}}
  &= \sqb{u, w, B_\rho\brk{a}}&
	&\text{in } \pi_m \brk{\manifold{N}, u \brk{a}} \eqpunct{.}
\end{align*}
Inserting an element of \(\sqb{u_\varepsilon, w, B_\rho\brk{a}}\) minimising for \eqref{eq_of7vae7eiquooDooh8nohnai}, we get 
\[
  \energy^{1, m}_{\mathrm{het}}
  \brk{u, v} 
  \le \int_{\manifold{M}} \abs{\Deriv u}^m
  +  \energy^{1, m}_{\mathrm{disp}} \brk{u, v} + \varepsilon\eqpunct{.}
\]
Since \( \varepsilon > 0 \) was arbitrary, the conclusion follows.

In the general case, if \(\brk{u_j}_{j \in \Nset}\) is a sequence in \( \continuous^{1} \brk{\manifold{M}, \manifold{N}} \) which converges to \(u\) in \(\sobolev^{1, m}\brk{\manifold{M}, \manifold{N}}\), we have by \cref{proposition_Ehet_lsc} and \cref{proposition_continuity_disparity}
\[
\begin{split}
  \energy^{1, m}_{\mathrm{het}}
  \brk{u, v} 
  &\le \liminf_{j \to \infty}
  \energy^{1, m}_{\mathrm{het}}
  \brk{u_j, v} \\
  &\le  \liminf_{j \to \infty }\,
  \brk[\bigg]{\int_{\manifold{M}} \abs{\Deriv u_j}^m
  +  \energy^{1, m}_{\mathrm{disp}} \brk{u_j, v}}
  = \int_{\manifold{M}} \abs{\Deriv u}^m
  +  \energy^{1, m}_{\mathrm{disp}} \brk{u, v}\eqpunct{,}
\end{split}
\]
which proves our claim.
\end{proof}

\section{Bubbling}

The core of this section is the following bubbling result, companion to~\cref{theorem_intro_heterotopic_concentration} as stated in the introduction, and which will be instrumental in the proof of the lower bound on the heterotopic energy.

\begin{theorem}
\label{theorem_concentration_measures}
Assume that \(v \in \continuous \brk{\manifold{M}, \manifold{N}}\).
If 
\begin{enumerate}[label=(\alph*)]
 \item for every \(j \in \Nset\), \(v_j \in \sobolev^{1, m} \brk{\manifold{M}, \manifold{N}}\) is homotopic to \( v \) in \(\VMO \brk{\manifold{M},\manifold{N}}\), 
\item\label{it_vj_to_u_L1} \(v_j \to u\) in \(\lebesgue^1 \brk{\manifold{M}, \manifold{N}}\), 
\item\label{it_Dvj_converges_measure} there exists a Radon measure \(\mu\) on \(\manifold{M}\)
such that for every \(\varphi \in \continuous \brk{\manifold{M}, \Rset}\),
\[
  \lim_{j \to \infty} \int_{\manifold{M}} \abs{\Deriv v_j}^m \varphi \dif x
  = \int_{\manifold{M}} \varphi \dif \mu\eqpunct{,}
\]
\end{enumerate}
then there exist points \( a_1, \dotsc, a_I \in \manifold{M} \) such that, for every \( \rho \in \intvo{0}{\infty} \) sufficiently small, there exists a map \(w \in \continuous \brk{\manifold{M}, \manifold{N}}\) homotopic to \(u\) such that
  \(w = v \) in \(\manifold{M} \setminus \bigcup_{i = 1}^I B_{\rho}\brk{a_i}\) and 
\begin{equation}
 \label{eq_CooF3fu9zahphohjahl2aelo}
\mu \ge \abs{\Deriv u}^m 
   + \sum_{i = 1}^I \mathfrak{E}^{1, m}_{\mathrm{top}}\brk{\sqb{w, v, B_\rho \brk{a_i}}}\, \delta_{a_i}\eqpunct{.}
\end{equation}
\end{theorem}

In other words, \eqref{eq_CooF3fu9zahphohjahl2aelo} states that if \(\varphi \in \continuous \brk{\manifold{M}, \intvr{0}{\infty}}\), then
\[
  \lim_{j \to \infty} \int_{\manifold{M}} \abs{\Deriv v_j}^m \varphi
  \ge \int_{\manifold{M}} \abs{\Deriv u}^m \varphi
  +  \sum_{i = 1}^I \mathfrak{E}^{1, m}_{\mathrm{top}}\brk{\sqb{w, v, B_\rho \brk{a_i}}}\,\varphi \brk{a_i}\eqpunct{.}
\]

In particular, the assumptions of \cref{theorem_concentration_measures} imply that \( u \in \sobolev^{1,m} \brk{\manifold{M}, \manifold{N}} \), as the weak limit of the mappings \( v_{j} \).

We also draw the attention of the reader to the fact that the relationships of the map \( w \) with the maps \( u \) and \( v \) are swapped between \cref{theorem_intro_heterotopic_concentration} in the introduction and \cref{theorem_concentration_measures} above.
The statement in \cref{theorem_intro_heterotopic_concentration} is somewhat more natural regarding the definition of the heterotopic energy, but as the map \( w \) provided by the statement is continuous, so needs to be as well the map with which it coincides outside of balls.
But, while it is not a loss of generality to assume \( v \) to be continuous, assuming \( u \) to be continuous would be quite a strong restriction to the general framework of low regularity maps that is studied in the body of this text.

As we already explained in the introduction, bubbling statements such as the above are ubiquitous in the study of weak convergence phenomena for Sobolev mappings.
The main contribution here is to state a very precise and general result, valid for any weakly converging sequence and relating precisely the limiting measure to the topological defect between the converging sequence and the limiting map, and to provide a complete proof of it.

 In view of \cref{lemma_disparity_trivial_homotopic}, we can assume that \(\sqb{w, v, B_\rho \brk{a_i}} \ne 0\);
 one has then by \cref{proposition_properties_Etop} \ref{it_Etop_gap0} a bound on the number of points where the bubbling phenomenon occurs, given by 
 \[
  I \le \frac{\mu\brk{\manifold{M}}}{\eta}
  = \frac{1}{\eta} \lim_{j \to \infty}\int_{\manifold{M}} \abs{\Deriv v_j}^m\eqpunct{.}
 \]
 
 Roughly speaking, the key idea behind the proof of \cref{theorem_concentration_measures} is to construct the desired map \( w \) by removing from \( v_j \), for \( j \) sufficiently large, the bubbles where the energy concentration occurs.

 A first tool for the proof of \cref{theorem_concentration_measures} is the following criterion for homotopies (see for example \cite{Hang_Lin_2003_II}*{Proof of Lemma 4.4}),
 which we apply to ensure that the map \( w \) that we construct is indeed homotopic to \( u \).

\begin{proposition}
\label{proposition_relative_homotopy}
There exists \(\eta \in \intvo{0}{\infty}\) such that, if \(\rho \in \intvo{0}{\infty}\) is sufficiently small, and if \(u, v \in \continuous \brk{\manifold{M}, \manifold{N}} \cap \sobolev^{1, m}\brk{\manifold{M}, \manifold{N}}\) satisfy
\begin{enumerate}[label=(\alph*)]
\item
\label{it_relative_homotopy_energy_balls}
for every \(a \in \manifold{M}\),
\begin{equation}
  \int_{B_\rho \brk{a}} \abs{\Deriv u}^m + \abs{\Deriv v}^m \le \eta\eqpunct{,}
\end{equation}
\item 
\label{it_relative_homotopy_distance_ball}
for every \(a \in \manifold{M}\),
\[
  \fint_{B_\rho \brk{a}} d\brk{u,v}^m \le \eta\eqpunct{,}
\]
\end{enumerate}
then \(u\) and \(v\) are homotopic.
\end{proposition}
\begin{proof}[Proof of \cref{proposition_relative_homotopy}]
The  argument follows the classical strategy of proof that goes back to Schoen and Uhlenbeck \cite{Schoen_Uhlenbeck_1983} (see also \cite{Brezis_Nirenberg_1995}), and we therefore only give a sketch of it.
Defining \(u_r, v_r \colon \manifold{M} \to \Rset^\nu\) for \(r \in \intvo{0}{\infty}\) sufficiently small by, for \(x \in \manifold{M}\),
\begin{align*}
 u_r \brk{x} &\defeq \fint_{B_r \brk{x}} u&
 &\text{and}&
 v_r \brk{x} &\defeq \fint_{B_r \brk{x}} v\eqpunct{,}
\end{align*}
one observes that \(u_r\) and \(v_r\) take values in a small tubular neighbourhood of the target \( \manifold{N} \) when \( r \le \rho\) thanks to \ref{it_relative_homotopy_energy_balls}.
Moreover, the condition \ref{it_relative_homotopy_distance_ball} ensures that \(u_\rho\) and \(v_\rho\) are uniformly close.
This shows that \(u\) and \(v\) are homotopic.
\end{proof}

We also use the following extension property for Sobolev mappings, in order to remove the bubbles formed by the weak convergence of the maps \( v_j \).

\begin{lemma}
\label{lemma_eta_extension}
There exists \(\eta \in \intvo{0}{\infty}\) such that if \(u_0 \in \sobolev^{1, m} \brk{\partial \Bset^m, \manifold{N}}\) satisfies 
\[
	\int_{\partial \Bset^m} \abs{\Deriv u_0}^m \le \eta\eqpunct{,}
\]
then there exists \(u \in \sobolev^{1, m} \brk{\Bset^m, \manifold{N}} \cap \continuous \brk{\Bset^m, \manifold{N}}\) such that \(\tr_{\partial \Bset^m} u = u_0\) and
\[
\int_{\Bset^m}\abs{\Deriv u}^m
\le C
 \int_{\partial \Bset^m} \abs{\Deriv u_0}^m\eqpunct{.}
\]
\end{lemma}

By a suitable scaling and exponential map construction, on every ball \( B_{\rho} \brk{a} \subset \manifold{M} \) of sufficiently small radius, \cref{lemma_eta_extension} shows that if
\[
 \rho\int_{\partial B_\rho \brk{a}} \abs{\Deriv u_0}^m \le \eta\eqpunct{,}
\]
then there exists \(u \in \sobolev^{1, m} \brk{B_\rho\brk{a}, \manifold{N}} \cap \continuous \brk{B_{\rho}\brk{a}, \manifold{N}}\) such that \(\tr_{\partial B_\rho\brk{a}} u = u_0\) and
 \[
 	 \int_{B_\rho \brk{a}} \abs{\Deriv u}^m
 	 \le C
 	 \rho\int_{\partial B_\rho \brk{a}} \abs{\Deriv u_0}^m
 	\eqpunct{,}
 \]
 with a constant \( C > 0 \) independent of \( \rho \).

\begin{proof}[Proof of \cref{lemma_eta_extension}]
 One first takes \(v \colon \Bset^m \to \Rset^\nu\) to be an extension of \(u_0\) by averages (for example a harmonic or hyperharmonic extension) and then applies a nearest-point projection.
For instance, following \cite{Petrache_VanSchaftingen_2017}*{\S 3}, one can set
 \[ 
 v \brk{x} \defeq  \brk{1 - \abs{x}^2}^{m - 1} \fint_{\Sset^{m - 1}} \frac{u_0 \brk{y}}{\abs{y - x}^{2 m - 2}} \dif y\eqpunct{;}
 \]
 and show that 
 \[
  \dist \brk{v \brk{x}, \manifold{N}}
  \le \C \brk[\Big]{\int_{\Sset^{m - 1}}\abs{\Deriv u_0}^{m - 1}}^\frac{1}{m - 1}
  \le \C \brk[\Big]{\int_{\Sset^{m - 1}}\abs{\Deriv u_0}^{m}}^\frac{1}{m}\eqpunct{.}
 \]
 One gets then the conclusion when \(\eta \in \intvo{0}{\infty}\) is sufficiently small.
 \resetconstant
 \end{proof}

 We rely on the following mixed Poincaré inequality to estimate the distance in \( \lebesgue^{m} \) between the map \( w \) that we construct and the map \( u \) on the balls where we perform the modification, in order to check that the assumptions of \cref{proposition_relative_homotopy} are satisfied.

 \begin{lemma}[Mixed Poincaré inequality]
 \label{lemma_mixed_poincare}
 If \(p \in \intvr{1}{\infty}\), then for every \(u \in \sobolev^{1, p}\brk{\Bset^m,\Rset^\nu}\), if \(\tr_{\partial \Bset^m} u = u\restr{\partial \Bset^m}\),
 one has 
 \[
  \smashoperator[r]{\iint_{\Bset^m \times \partial \Bset^m}}
  \abs{u \brk{x} - u \brk{y}}^p\dif y \dif x
  \le C \int_{\Bset^m} \abs{\Deriv u}^p\eqpunct{.}
 \]
 \end{lemma}
 By a straightforward scaling and local chart argument, on every ball \( B_{\rho} \brk{a} \subset \manifold{M} \) of sufficiently small radius,~\cref{lemma_mixed_poincare} implies that
 \[
 	\fint_{B_{\rho} \brk{a}}\fint_{\partial B_{\rho} \brk{a}}
 	\abs{u \brk{x} - u \brk{y}}^p\dif y \dif x
 	\le C \rho^{p-m} \int_{B_{\rho} \brk{a}} \abs{\Deriv u}^p\eqpunct{,}
 \]
 with a constant \( C > 0 \) independent of \( \rho \).
 \begin{proof}[Proof of \cref{lemma_mixed_poincare}]
 	Let us first note that it suffices to prove the statement for \( u \in \continuous^{\infty} \brk{\overline{\Bset^{m}}, \Rset^{\nu}} \).
 	Indeed, the general case then follows by approximation, relying on the continuity of the trace operator \( \sobolev^{1,p}\brk{\Bset^{m}} \to \lebesgue^{p}\brk{\partial\Bset^{m}} \).
 	Let us hence assume that \( u \in \continuous^{\infty} \brk{\overline{\Bset^{m}}, \Rset^{\nu}} \).
 	
	For a.e.\ \(y \in \partial \Bset^m\), combining the mean value inequality, Jensen's inequality, and Fubini's theorem, we have
	\[
	\begin{split}
	  \int_0^1 \abs{u \brk{y} - u \brk{ry}}^p r^{m - 1}\dif r 
	  &\le \int_0^1 \brk[\bigg]{\int_r^1 \abs{\nabla u \brk{sy}} \dif s}^pr^{m - 1}\dif r\\
	  &\le \int_0^1 \int_r^1 \abs{\nabla u \brk{sy}}^p \brk{1 - r}^{p - 1} r^{m - 1}\dif s \dif r\\
	  &\le \int_0^1  \abs{\nabla u \brk{sy}}^p \int_0^s \brk{1 - r}^{p - 1} r^{m - 1}\dif r \dif s\\
	  &\le \frac{1}{m} \int_0^1 \abs{\Deriv u \brk{r y}}^p r^m \dif r
	  \eqpunct{,}
	\end{split}
	\]
	and thus by spherical integration
	\begin{equation}
	\label{eq_aek2Oo8shaawei7Thoiz2ohC}
	 \int_{\Bset^m} \abs{u \brk{x/\abs{x}} - u\brk{x}}^p \dif x
	 \le \frac{1}{m}\int_{\Bset^m} \abs{x} \abs{\Deriv u\brk{x}}^p \dif x\eqpunct{.}
	\end{equation}
	On the other hand, by the Poincaré inequality we have 
	\begin{equation}
	\label{eq_ooreekiduCheshef3baithai}
	 \iint_{\Bset^m \times \Bset^m} \abs{u \brk{x} - u\brk{y}}^p \dif x \dif y
	 \le \C \int_{\Bset^m} \abs{\Deriv u}^p\eqpunct{.}
	\end{equation}
	Combining \eqref{eq_aek2Oo8shaawei7Thoiz2ohC} and \eqref{eq_ooreekiduCheshef3baithai} with the triangle inequality, we get 
	\[
	\begin{split}
	    &\smashoperator[r]{\iint_{\partial \Bset^m \times \Bset^m}}
	  \abs{u \brk{x} - u \brk{y}}^p\dif y \dif x\\
	  &\qquad = m
	  \smashoperator{\iint_{\Bset^m \times \Bset^m}}
	  \abs{u \brk{x/\abs{x}} - u \brk{y}}^p\dif y \dif x\\
	  &\qquad \le 
	  \C \brk[\bigg]{\smashoperator[r]{\iint_{\Bset^m \times \Bset^m}}
	  \abs{u \brk{x/\abs{x}} - u \brk{x}}^p\dif y \dif x+
	   \smashoperator{\iint_{\Bset^m \times \Bset^m}}
	  \abs{u \brk{x} - u \brk{y}}^p\dif y \dif x 
	  }\\
	  &\qquad \le \C \int_{\Bset^m} \abs{\Deriv u}^p\eqpunct{.}\qedhere
	\end{split}
	\]
\resetconstant
\end{proof}

A last tool is a measure-theoretical lemma that describes the concentration of measures on balls.
\begin{lemma}
\label{lemma_measure_concentration}
Let \(\mu\) be a finite Borel measure on \(\manifold{M}\), \(\eta \in \intvo{0}{\infty}\), and let
\[
 A\defeq \set{a \in X\st \mu \brk{a} \ge \eta}
 \eqpunct{.}
\]
If \(\rho \in \intvo{0}{\infty} \) is sufficiently small, then for every \(x \in \manifold{M}\setminus \bigcup_{a \in A} B_{\rho}\brk{a}\), one has
\[
 \mu \brk{\Bar{B}_{\rho/2}\brk{x}} < \eta
 \eqpunct{.}
\]
\end{lemma}
\begin{proof}
Assume by contradiction that there is a sequence \(\brk{\rho_j}_{j \in \Nset}\) in \(\intvo{0}{\infty}\) converging to \(0\) and a sequence \(\brk{x_j}_{j \in \Nset}\) in \(\manifold{M}\) such that \(x_j \in \manifold{M}\setminus \bigcup_{a \in A} B_{\rho_j}\brk{a}\) and
\[
  \mu \brk{\Bar{B}_{\rho_j/2}\brk{x_j}} \ge \eta
  \eqpunct{.}
\]
Since \(\manifold{M}\) is compact, we can assume that \(\brk{x_j}_{j \in \Nset}\) converges to some \(a \in \manifold{M}\).

Given \(\delta > 0\), if \(j \in \Nset\) is sufficiently large,
\(\Bar{B}_{\rho_j/2} \brk{x_j} \subseteq \Bar{B}_{\delta}\brk{a}\), and thus
\[
 \mu \brk{\Bar{B}_{\delta}\brk{a}} \ge \mu \brk{\Bar{B}_{\rho_j/2}\brk{x_j}} \ge \eta\eqpunct{.}
\]
Letting \(\delta \to 0\), \(\mu \brk{\set{a}} \ge \eta \) and thus \(a \in A\).

By assumption, we have
\[
 \rho_j \le \abs{x_j - a}
 \eqpunct{.}
\]
Taking a further subsequence if necessary, we can assume that
\[
3 \abs{x_{j + 1} - a} < \abs{x_j - a}
\eqpunct{.}
\]
It follows then that if \(k > j\),
\[
\begin{split}
 \abs{x_j - x_{k}}
 &\ge \abs{x_j - a} - \abs{x_{k} - a}\\
 &= \frac{\abs{x_j - a} + \abs{x_{k} - a}}{2} + \frac{\abs{x_j - a} - 3  \abs{x_{k} - a}}{2}\\
 &> \frac{\rho_j + \rho_{k}}{2}\eqpunct{,}
 \end{split}
\]
and thus \(\Bar{B}_{\rho_j/2}\brk{x_j} \cap \Bar{B}_{\rho_{k}/2}\brk{x_{k}} = \varnothing\).
We have then
\[
 \mu \brk{\manifold{M}} \ge
 \sum_{j \in \Nset} \mu \brk{\Bar{B}_{\rho_j/2} \brk{x_j}}
 = \infty\eqpunct{,}
\]
in contradiction with the finiteness of the measure \(\mu\) on \(\manifold{M}\).
\end{proof}

We now have all the tools at our disposal in order to prove the main result of this section.

\begin{proof}[Proof of \cref{theorem_concentration_measures}]
By a classical lower semi-continuity argument, we have for every \(\varphi \in \continuous \brk{\manifold{M}, \intvr{0}{\infty}}\)
\[
 \int_{\manifold{M}} \varphi \dif \mu = 
 \lim_{j \to \infty} \int_{\manifold{M}} \abs{\Deriv v_{j}}^m \varphi
 \ge \int_{\manifold{M}} \abs{\Deriv u}^m \varphi\eqpunct{,}
\]
so that, as measures, 
\begin{equation}
\mu \ge \abs{\Deriv u}^m\eqpunct{.}
\end{equation}

We now apply an approximation argument, in order to be able to work instead with smooth maps.
More specifically, by the strong density of smooth maps in \(\sobolev^{1, m} \brk{\manifold{M}, \manifold{N}}\), there exist sequences \(\brk{\Hat{u}_j}_{j \in \Nset}\)
 and \(\brk{\Hat{v}_j}_{j \in \Nset}\) in \(\continuous^{\infty}\brk{\manifold{M}, \manifold{N}}\) such that 
\begin{equation}
\label{eq_7a050b9872b90f0e}
 \lim_{j \to \infty} \int_{\manifold{M}} \abs{\Deriv \Hat{u}_j - \Deriv u}^m + d\brk{\Hat{u}_j, u}^{m} = 0
\end{equation}
and 
\begin{equation}
\label{eq_d66944cda8ad35b8}
 \lim_{j \to \infty} \int_{\manifold{M}} \abs{\Deriv \Hat{v}_j - \Deriv v_j}^m + d\brk{\Hat{v}_j, v_j}^{m} = 0\eqpunct{.}
\end{equation}
It follows from \cref{proposition_stability_homotopy_classes} that for each \(j \in \Nset\) sufficiently large, the map \(\Hat{u}_j\) is homotopic to \(u\) in \(\VMO \brk{\manifold{M}, \manifold{N}}\) whereas the map \(\Hat{v}_j\) is homotopic to \( v_j \) in \(\continuous \brk{\manifold{M}, \manifold{N}}\).

After these preliminaries, we are at the heart of the proof of \cref{theorem_concentration_measures}.
We first study the points where energy concentration occurs.
For some \(\eta \in \intvo{0}{\infty}\) to be fixed, we define the set
\[
  \set{a_1, \dotsc, a_I} = \set{a \in \manifold{M} \st \mu \brk{\set{a}} \ge \eta}\eqpunct{.}
\]
If \(\rho \in \intvo{0}{\infty}\) is chosen sufficiently small according to \cref{lemma_measure_concentration}, then if \(j \in \Nset\) is sufficiently large, 
for every \(x \in \manifold{M}\setminus \bigcup_{i = 1}^I B_{\rho} \brk{a_i}\), we have
\[
 \int_{B_{\rho/2} \brk{x}} \abs{\Deriv v_j}^m < \eta\eqpunct{.}
\]
Indeed, assume that there is a sequence \(\brk{x_j}_{j \in \Nset}\) such that
\(x_j \in \manifold{M} \setminus \bigcup_{i = 1}^I B_{\rho} \brk{a_i}\) and such that
\[
 \int_{B_{\rho/2} \brk{x_j}} \abs{\Deriv v_j}^m \ge \eta\eqpunct{;}
\]
in view of the compactness of \(\manifold{M}\),
we can assume up to a subsequence that \(\brk{x_j}_{j \in \Nset}\) converges to some \(x_* \in \manifold{M}\setminus \smash{\bigcup_{i = 1}^I} B_{\rho} \brk{a_i}\);
if \(\sigma > \rho\), then
\[
 \mu \brk{\Bar{B}_{\sigma/2}\brk{x_*}}
 \ge \limsup_{j \to \infty} \int_{B_{\sigma/2} \brk{x_*}} \abs{\Deriv v_j}^m
 \ge \limsup_{j \to \infty} \int_{B_{\rho/2} \brk{x_j}} \abs{\Deriv v_j}^m
 \ge \eta\eqpunct{,}
\]
so that 
\[
\mu \brk{\Bar{B}_{\rho/2}\brk{x_*}} = \lim_{\sigma \underset{>}{\to} \rho} \mu \brk{\Bar{B}_{\sigma/2}\brk{x_*}}
 \ge \eta\eqpunct{.}
\]
In view of the definition of the set \(\set{a_1, \dotsc, a_I}\) and of \cref{lemma_measure_concentration}, this cannot hold when \(\rho\) is sufficiently small.

If \(j \in \Nset\) is sufficiently large, we also obtain from \eqref{eq_d66944cda8ad35b8} that for every \(x \in \manifold{M}\setminus \smash{\bigcup_{i = 1}^I} B_\rho \brk{a_i}\),
\begin{equation}
\label{eq_ba96089e7b05d55b}
  \int_{B_{\rho/2} \brk{x}} \abs{\Deriv \Hat{v}_j}^m < \eta\eqpunct{.}
\end{equation}
Concerning the energy around the concentration points, it is governed by the measure \( \mu \).
More specifically, given \(\varepsilon > 0\), using~\eqref{eq_d66944cda8ad35b8} and assumption~\ref{it_Dvj_converges_measure}, we may assume that \(j\) is sufficiently large so that
\begin{equation}
\label{eq_fee8Ech5fae2oph0ahMea6oh}
\int_{B_{\rho} \brk{a_i}} \abs{\Deriv \Hat{v}_j}^m
\le 
\mu \brk{\Bar{B}_{\rho}\brk{a_i}}
+ \varepsilon\eqpunct{.}
\end{equation}

On the contrary, concerning the maps \( \Hat{u}_j \), we may have a small energy estimate on \emph{every} ball.
Indeed, using~\eqref{eq_7a050b9872b90f0e} and Vitali's convergence theorem, if \( \rho \in \intvo{0}{\infty} \) is sufficiently small and \( j \in \Nset \) is sufficiently large, then for every \( x \in \manifold{M}\),
\begin{equation}
\label{eq_7732e2d014b435e5}
  \int_{B_{\rho} \brk{x}} \abs{\Deriv \Hat{u}_j}^m < \eta\eqpunct{.}
\end{equation}

We now wish to construct our desired map \( w \) by removing the bubbles on the balls \( B_{\rho} \brk{a_i} \) by using \cref{lemma_eta_extension}.
For this purpose, we first find a suitable radius \( \rho \) to work on.
By a Fubini-type argument, relying on the assumption~\ref{it_vj_to_u_L1} and using~\eqref{eq_7a050b9872b90f0e} and~\eqref{eq_d66944cda8ad35b8} again, we can assume that 
\begin{equation}
\label{eq_1685b15d400431fc}
\lim_{j \to \infty} \int_{\partial B_\rho\brk{a_i}} d\brk{\Hat{u}_j, \Hat{v}_j}^{m} = 0\eqpunct{.}
\end{equation}
Moreover, by Fatou’s lemma and Fubini’s theorem, it also holds that
\[
\begin{split}
 \int_{0}^{\rho}\brk[\Big]{\liminf_{j \to \infty}r \int_{\partial B_r \brk{a_i}} \abs{\Deriv \Hat{v}_j}^m} \frac{\dif r}{r}
 &\le \liminf_{j \to \infty}\int_{0}^{\rho}\brk[\Big]{\int_{\partial B_r \brk{a_i}} \abs{\Deriv \Hat{v}_j}^m} \dif r \\
 & \le \C \liminf_{j \to \infty} \int_{\manifold{M}} \abs{\Deriv v_j}^m < \infty\eqpunct{,}
\end{split}
\]
so that we can assume that
\[
  \liminf_{j \to \infty}   \rho \int_{\partial B_\rho \brk{a_i}} \abs{\Deriv \Hat{v}_j}^m \le \varepsilon \le \eta\eqpunct{.}
\]

Applying \cref{lemma_eta_extension} to every ball \(B_{\rho} \brk{a_i}\), we get a map \(\Tilde{u}_j \in \sobolev^{1, m} \brk{\manifold{M}, \manifold{N}}\cap \continuous \brk{\manifold{M}, \manifold{N}}\) such that
\(\Tilde{u}_j = \Hat{v}_j\) in \(\manifold{M}\setminus \smash{\bigcup_{i = 1}^I B_{\rho}\brk{a_i}}\) and
\begin{equation}
\label{eq_2d0a57078693cf23}
 \int_{B_\rho\brk{a_i}} \abs{\Deriv \Tilde{u}_j}^m
 \le \C \rho  \int_{\partial B_\rho\brk{a_i}} \abs{\Deriv \Hat{v}_j}^m
 \le \C \varepsilon\eqpunct{.}
\end{equation}
It follows then that for every \(a \in \manifold{M}\),
\begin{equation*}
 \int_{B_{\rho/2} \brk{a}} \abs{\Deriv \Tilde{u}_j}^m
 \le \C \eta\eqpunct{.}
\end{equation*}

By the mixed Poincaré inequality (\cref{lemma_mixed_poincare}), we have
\[
\begin{split}
 \fint_{B_\rho \brk{a_i}}
 d \brk{\Hat{u}_j, \Tilde{u}_j}^m
 &\le
 \C
 \brk[\bigg]{\fint_{\partial B_\rho \brk{a_i}} d \brk{\Hat{u}_j, \Tilde{u}_j}^m\\
 &\qquad\qquad
 + \fint_{B_\rho \brk{a_i}} \fint_{\partial B_\rho\brk{a_i}}
 d \brk{\Hat{u}_j \brk{x}, \Hat{u}_j \brk{y}}^m +
 d \brk{\Tilde{u}_j \brk{x}, \Tilde{u}_j \brk{y}}^m\dif x\dif y}\\
 &\le \C
 \brk[\Big]{\fint_{\partial B_\rho \brk{a_i}} d \brk{\Hat{u}_j, \Hat{v}_j}^m
 + \int_{B_\rho \brk{a_i}} \abs{\Deriv \Hat{u}_j}^m + \abs{\Deriv \Tilde{u}_j}^m}\eqpunct{.}
\end{split}
\]
If \(j\) is sufficiently large, relying on~\eqref{eq_1685b15d400431fc},~\eqref{eq_7732e2d014b435e5}, and~\eqref{eq_2d0a57078693cf23}, we have
\[
  \fint_{B_\rho \brk{a_i}}
 d \brk{\Hat{u}_j, \Tilde{u}_j}^m
 \le \C \eta\eqpunct{.}
\]
In addition, if \( j \) is sufficiently large (depending on \( \rho \)), then
\[
 \frac{1}{\rho^{m}}\int_{\manifold{M}\setminus \bigcup_{i = 1}^I B_{\rho}\brk{a_i}} d \brk{\Hat{u}_j, \Tilde{u}_j}^{m}
 =
 \frac{1}{\rho^{m}}\int_{\manifold{M}\setminus \bigcup_{i = 1}^I B_{\rho}\brk{a_i}} d \brk{\Hat{u}_j, \Hat{v}_j}^{m}
 \le \eta\eqpunct{.}
\]
We are therefore in position to apply \cref{proposition_relative_homotopy} and conclude that \(\Tilde{u}_j\) and \(\Hat{u}_j\) are homotopic.

We now turn to the estimate of the disparity energy on the balls \( B_{\rho} \brk{a_i} \).
By \cref{proposition_estimate_E1mtop_u_v}, we have then for every \(i \in I\), in view of \eqref{eq_fee8Ech5fae2oph0ahMea6oh} and~\eqref{eq_2d0a57078693cf23},
\begin{equation}
\label{eq_ieboosuquoo6sei4aziwoWae}
\begin{split}
\mathfrak{E}^{1, m}_{\mathrm{top}}\brk{\sqb{\Tilde{u}_j, \Hat{v}_j, B_\rho \brk{a_i}}}
  &\le \brk{1 + \Cl{cst_736bda0cd15cb00a} \rho} \int_{B_{\rho} \brk{a_i}} \abs{\Deriv \Hat{v}_j}^m + \abs{\Deriv \Tilde{u}_j}^m \\
  &\le \mu \brk{\Bar{B}_\rho\brk{a_i}} + \
  \C \brk{\rho + \varepsilon}
  \eqpunct{.}
  \end{split}
\end{equation}
The factor \( 1 + \Cr{cst_736bda0cd15cb00a}\rho \) in~\eqref{eq_ieboosuquoo6sei4aziwoWae} accounts for the distortion when identifying \( \Bar{B}_\rho \brk{a_i} \) with the Euclidean unit ball \( \overline{\Bset^{m}} \).
Since the range of \(\mathfrak{E}^{1, m}_{\mathrm{top}}\) is discrete (\cref{proposition_properties_Etop} \ref{it_Etop_discrete}), we have for some \(\rho\) sufficiently small
\[
 \mathfrak{E}^{1, m}_{\mathrm{top}}\brk{\sqb{\Tilde{u}_j, \Hat{v}_j, B_\rho \brk{a_i}}}
 \le \mu \brk{a_i}\eqpunct{.}
\]

The map \(w\) is then defined thanks to \cref{lemma_discrepancy_homotopy_transport}, using that \( \Hat{v}_j = \Tilde{u}_j \) in \(\manifold{M}\setminus \bigcup_{i = 1}^I B_{\rho}\brk{a_i}\) and that \( v \) is homotopic to \( \Hat{v}_j \) for \( j \) sufficiently large; it satisfies the conclusion in view of \cref{proposition_properties_Etop} \ref{it_Etop_action}.
\resetconstant
\end{proof}

\begin{proof}[Proof of \cref{theorem_intro_heterotopic_concentration}]
We apply \cref{theorem_concentration_measures}; relying on \cref{lemma_discrepancy_homotopy_transport}, we may assume that instead \(w = u\) in \(\manifold{M} \setminus \bigcup_{i = 1}^I B_\rho\brk{a_i}\), \(w\) is homotopic to \(v\), and
\begin{equation}
\label{eq_CooF3fu9zahphohjahl2aelobis}
	\mu \ge \abs{\Deriv u}^m 
	+ \sum_{i = 1}^I \mathfrak{E}^{1, m}_{\mathrm{top}}\brk{\sqb{u, w, B_\rho \brk{a_i}}}\,\delta_{a_i}\eqpunct{,}
\end{equation}
where we have used \cref{proposition_properties_Etop}~\ref{it_Etop_action} and \cref{lemma_disparity_reciprocal}.
A standard smoothing argument allows to obtain the additional regularity \( w \in \continuous^{\infty} \brk{\manifold{M}, \manifold{N}} \).

The inequality \eqref{eq_iasooYu9pohhuf7Jai0Eghoo} follows from \eqref{eq_CooF3fu9zahphohjahl2aelobis} and properties of the convergence of measures;  \eqref{eq_quivu5Dai3Oongis9ohPooch} follows then from \cref{proposition_Edis_controlled_several_bubbles}.
\end{proof}

Let us observe that the above proof shows that the conclusion of \cref{theorem_intro_heterotopic_concentration} holds under the weaker assumption that \( u \in \sobolev^{1,m} \brk{\manifold{M}, \manifold{N}} \cap \continuous \brk{\manifold{M}, \manifold{N}} \) and \( v \in \continuous \brk{\manifold{M}, \manifold{N}} \); in this case, the map \( w \) that we obtain is only continuous, not smooth.

We are now in position to conclude the proof of \cref{theorem_intro_heterotopic_disparity} -- and even a low regularity version of it -- by proving the lower bound on the heterotopic energy.

\begin{theorem}
\label{theorem_characterisation_Heten}
If \(u \in \sobolev^{1, m}\brk{\manifold{M}, \manifold{N}} \cap \continuous\brk{\manifold{M}, \manifold{N}} \) and \(v \in \continuous \brk{\manifold{M}, \manifold{N}}\), then 
\[
 \energy^{1, m}_{\mathrm{het}}
  \brk{u, v}
  = \int_{\manifold{M}} \abs{\Deriv u}^m + \energy^{1, m}_{\mathrm{disp}} \brk{u, v}\eqpunct{.}
\]
\end{theorem}
\begin{proof}
This follows from \cref{proposition_upper_estimate} and  \cref{theorem_intro_heterotopic_concentration} (or more precisely, its lower regularity counterpart, see the comment following the proof of \cref{theorem_intro_heterotopic_concentration}).
\end{proof}

We conclude with the following statement concerning the continuity of the heterotopic energy with respect to the strong \( \sobolev^{1,m} \) convergence.

\begin{proposition}
If \(\brk{u_j}_{j \in \Nset}\) is a sequence in \(\sobolev^{1, m}\brk{\manifold{M}, \manifold{N}}\) converging strongly to \(u \in \sobolev^{1, m} \brk{\manifold{M}, \manifold{N}}\), then 
\[
 \energy^{1, m}_{\mathrm{het}}
  \brk{u, v}
  = \lim_{j \to \infty} \energy^{1, m}_{\mathrm{het}}
  \brk{u_j, v}\eqpunct{.}
\]
\end{proposition}
\begin{proof}
This follows from \cref{theorem_characterisation_Heten} and \cref{proposition_continuity_disparity}.
\end{proof}

\end{document}